\documentclass[15pt]{amsart}
\usepackage{amsmath}
\usepackage{mathtools}
\usepackage{}
\usepackage{graphicx}
\usepackage[colorlinks=true, allcolors=blue]{hyperref}
\usepackage{amsfonts}
\usepackage{amsthm}
\usepackage{newlfont}
\usepackage{amscd}
\usepackage{amsgen}
\usepackage{amssymb}
\usepackage{mathrsfs}	
\usepackage{longtable}
\usepackage{listings}
\usepackage{extarrows}
\usepackage{tikz}
\usepackage{tikz-cd}
\usepackage{verbatim}
\numberwithin{equation}{section}
\usepackage[all]{xy}
\usepackage{color}
\usepackage{amssymb}
\usepackage[left=3cm, right=3cm]{geometry}
\usepackage{tikz-cd}
\usepackage{mathtools}
\usepackage{dynkin-diagrams}
\usetikzlibrary{matrix,shapes,arrows,decorations.pathmorphing}
\usepackage{calligra}
\usepackage{mathrsfs}
\usepackage{enumitem} 
\usepackage{tgpagella} 
\usepackage[parfill]{parskip} 
\usepackage{float}
\restylefloat{table}
\usepackage{ytableau}

\let\blb\mathbb

\def\CC{{\blb C}}

\def\II{{\blb I}}

\def\LL{{\blb L}}

\def\PP{{\blb P}}
\def\QQ{{\blb Q}}

\def\ZZ{{\blb Z}}

\def\RR{{\blb R}}

\let\cal\mathcal

\def\Bc{{\cal B}}

\def\Ec{{\cal E}}
\def\Fc{{\cal F}}
\def\Gc{{\cal G}}
\def\Hc{{\cal H}}
\def\Ic{{\cal I}}

\def\Lc{{\cal L}}

\def\Oc{{\cal O}}

\def\Qc{{\cal Q}}

\def\Sc{{\cal S}}
\def\Tc{{\cal T}}
\def\Uc{{\cal U}}

\def\Wc{{\cal W}}
\def\Xc{{\cal X}}

\DeclareMathOperator{\Gr}{Gr}
\DeclareMathOperator{\Fl}{Fl}
\newcommand{\mQ}{\mathcal{Q}}
\newcommand{\mU}{\mathcal{U}}

\DeclareMathOperator{\Sym}{Sym}
\newcommand{\of}{\mathcal{O}}
\newcommand{\W}{\bigwedge}

\newcommand{\mF}{\mathcal{F}}
\newcommand{\vgit}{\operatorname{\mathbf v}}

\newtheorem{lemma}{Lemma}[section]
\newtheorem{proposition}[lemma]{Proposition}
\newtheorem{theorem}[lemma]{Theorem}
\newtheorem{corollary}[lemma]{Corollary}

\newtheorem{definition}[lemma]{Definition}
\newtheorem{conjecture}[lemma]{Conjecture}

\newtheorem{todo}[lemma]{To do}

\theoremstyle{remark}

\newtheorem{remark}[lemma]{Remark}

\def\dbcoh{D^b}
\def\wt{\widetilde}
\def\wh{\widehat}
\def\arw{\longrightarrow}
\def\Hom{\operatorname{Hom}}

\def\im{\operatorname{im}}

\def\Ext{\operatorname{Ext}}
\def\rk{\operatorname{rk}}
\def\crit{\operatorname{Crit}}
\def\git{/\hspace{-3pt}/}

\DeclareMathOperator{\sHom}{\mathscr{H}\text{\kern -3pt {\calligra\large om}}\,}

\newcommand\quotient[2]{
        \mathchoice
            {
                \text{\raise1ex\hbox{$#1$}\Big/\lower1ex\hbox{$#2$}}%
            }
            {
                #1\,/\,#2
            }
            {
                #1\,/\,#2
            }
            {
                #1\,/\,#2
            }
    }

\makeatletter
\def\namedlabel#1#2{\begingroup
    #2%
    \def\@currentlabel{#2}%
    \phantomsection\label{#1}\endgroup
}
\makeatother

\makeatletter
\newcommand{\xdashrightarrow}[2][]{\ext@arrow 0359\rightarrowfill@@{#1}{#2}}
\newcommand{\xdashleftarrow}[2][]{\ext@arrow 3095\leftarrowfill@@{#1}{#2}}
\newcommand{\xdashleftrightarrow}[2][]{\ext@arrow 3359\leftrightarrowfill@@{#1}{#2}}
\def\rightarrowfill@@{\arrowfill@@\relax\relbar\rightarrow}
\def\leftarrowfill@@{\arrowfill@@\leftarrow\relbar\relax}
\def\leftrightarrowfill@@{\arrowfill@@\leftarrow\relbar\rightarrow}
\def\arrowfill@@#1#2#3#4{%
  $\m@th\thickmuskip0mu\medmuskip\thickmuskip\thinmuskip\thickmuskip
   \relax#4#1
   \xleaders\hbox{$#4#2$}\hfill
   #3$%
}
\makeatother

\title{The generalized roof $F(1,2,n)$: Hodge structures and derived categories}

\author{Enrico Fatighenti}
\address{
Sapienza - Universit\`a di Roma\\ Dipartimento di Matematica "Guido Castelnuovo"\\ Piazzale Aldo Moro 5\\ 00185 Roma.}
\email[E.~Fatighenti]{fatighenti@mat.uniroma1.it}

\author{Micha\l\ Kapustka}
\address{
Institute of Mathematics of the Polish Academy of Sciences \\
ul. Śniadeckich 8, 00-656 Warsaw.}
\email[M.~ Kapustka]{michal.kapustka@impan.pl}

\author{Giovanni Mongardi}
\address{ Alma Mater Studiorum Università di Bologna\\ Dipartimento di Matematica\\ Piazza di Porta San Donato 5\\ 40126 Bologna.}
\email[G.~Mongardi]{giovanni.mongardi2@unibo.it}

\author{Marco Rampazzo}
\address{
Alma Mater Studiorum Università di Bologna\\ Dipartimento di Matematica \\ Piazza di Porta San Donato 5\\ 40126 Bologna.}
\email[M.~ Rampazzo]{marco.rampazzo3@unibo.it}

\begin{document}

\maketitle

\begin{abstract}
    We consider generalized homogeneous roofs, i.e. quotients of simply connected, semisimple Lie groups by a parabolic subgroup, which admit two projective bundle structures. Given a general hyperplane section on such a variety, we study the zero loci of its pushforwards along the projective bundle structures and we discuss their properties at the level of Hodge structures. In the case of the flag variety $F(1,2,n)$ with its projections to $\PP^{n-1}$ and $G(2, n)$, we construct a derived embedding of the relevant zero loci by methods based on the study of $B$-brane categories in the context of a gauged linear sigma model.
\end{abstract}

\subsection*{Keywords}
14J45 Fano varieties; 14J81 relationship with physics; 14F08  derived categories of sheaves, dg categories, and related constructions in algebraic geometry; 14C30 transcendental methods, Hodge theory (algebro-geometric aspects); 14M15 Grassmannians, Schubert varieties, flag manifolds.

\section{Introduction}
The existence of two different projective bundle structures on a Fano variety has recently raised attention: in \cite{occhettaetaltwoproj} a link with the Campana--Peternell conjecture has been highlighted, while in \cite{kanemitsu} the study of a broad class of K-equivalences has been reduced to the classification of \emph{roofs} of projective bundles, i.e. special Fano varieties of Picard rank two which admit two different projective bundle structures of the same rank. In \cite{kr2, roofbundles} this idea has been related to the construction of pairs of Calabi--Yau varieties and their Hodge and derived equivalences. In particular, when the roof is a rational homogeneous variety $G/P$, the construction is greatly simplified: in this class fall, for instance, the derived and Hodge equivalent Calabi--Yau pairs of \cite{mukaiduality, imou, kuznetsovimou, kr, roofbundles}.\\
\\
In this paper, we introduce a generalization of Kanemitsu's notion of roof: a \emph{generalized roof} is a rational homogeneous variety of Picard rank two which admits two projective bundle structures and a line bundle $\Lc$ restriciting to $\Oc(1)$ on both types of fibers. The original definition of roof is achieved by asking for the two projective bundles to have the same rank. We then consider generalized roofs which are homogeneous varieties and call these varieties generalized homogeneous roofs.
Let $X$ be a generalized homogeneous roof with projective bundle structures $h_i:X\arw B_i$ for $i\in\{1;2\}$. Given a general section of $\Lc$, which in the homogeneous case is  $\Oc(1,1):=h_1^*\Oc(1)\otimes h_2^*\Oc(1)$, we construct a pair of varieties $Y_i\subset B_i$ defined as the zero loci of the two pushforwards of that section along $h_1$ and $h_2$. While in the case of a roof these varieties are Calabi--Yau, the vanishing of the first Chern class is in general lost by our generalization. Instead, we provide examples of pairs where $Y_1$ is general type and $Y_2$ is Fano.

\subsubsection*{Hodge-theoretical results: a Fano/general type duality}
One of the main results of this paper is the identification of countably many families of varieties with very different geometrical properties, but the same \emph{Hodge-theoretical core}. In fact, our starting point was to notice a surprising coincidence between the Hodge numbers of the zero loci of (a general global section of) the vector bundles $\mQ^{\vee}(2)$ and $\mU(2)$ on the Grassmannians $ G (k, V)$ and $ G (k+1, V)$. These zero loci are subcanonical, with canonical class opposite in sign and equal in modulus. In particular, for any fixed  {$k$}, they cut either a pair of a (smaller dimensional)  general type variety and a (higher dimensional) Fano variety, or a pair of Calabi-Yau varieties. In the Calabi-Yau case the Hodge numbers are exactly the same. In fact this can be considered a generalization of the $ G (2,5)$ case, already studied in details by \cite{kr}, \cite{bpc}, \cite{or}. Perhaps more surprisingly, a similar phenomenon holds in the general type/Fano scenario, with the (relevant) non-zero Hodge numbers of the latter being equal to the (relevant) Hodge numbers of the former.
In order to explain this phenomenon, we give a slightly modified version of the \emph{jump} procedure in \cite{bfm}, and produce an isomorphism of a (rational) Hodge sub-structure which coincides with the vanishing cohomology in the known cases. This is the content of Theorem \ref{thm:hodge}. In the Calabi-Yau case, it is then natural to conjecture that (as in the $ G (2,5)$ case) such a pair of varieties are derived equivalent, but not isomorphic. In the general type/Fano case, our result can be considered as a Hodge-theoretical approximation of the \emph{Fano visitor} phenomenon, in the sense of, e.g., \cite{kkll}. In fact, although we prove the derived categorical extension of Theorem \ref{thm:hodge} only when $k=1$, it is natural to speculate that the general type variety of the pair will be a \emph{visitor}, and the Fano the corresponding \emph{Host}. In an upcoming work we will go even further in our speculation. We study a more general conjecture formulated in terms of homological projective duality, and we support it with some examples.

\subsubsection*{Generalized roofs and derived categories}
In \cite{kr2} it has been conjectured that a pair of Calabi--Yau varieties arising from a roof should be derived equivalent, and such conjecture is supported by many examples \cite{mukaiduality, kuznetsovimou, kr, kr2}. It is natural to expect that in our generalization a derived embedding should occur. We prove the existence of such embedding for the case of the generalized roof $F(1,2, n)$ for every $n$ (Theorem \ref{thm:derived_embedding_main}). The proof relies on the method of \emph{window categories} (see for example \cite{addingtondonovansegal, segal}), i.e. identifying the derived categories of $Y_1$ and $Y_2$ with some categories of $B$-branes via Kn\"orrer periodicity \cite{shipman}, and constructing the embedding at the level of $B$-branes. A possible, alternative proof of Theorem \ref{thm:derived_embedding_main} could have been formulated in the spirit of Leung--Xie's approach to the related problem of the flip between total spaces \cite{leungxie}, see Remark \ref{rem:chessgame} for the relation with our proof.

\subsubsection*{Interaction with physics} The proof of Theorem \ref{thm:derived_embedding_main} requires the construction of a gauged linear sigma model (GLSM). Such objects have been first introduced by Witten in \cite{witten} to explain the correspondence between two different quantum field theories via phase transition: a nonlinear sigma model with a Calabi--Yau complete intersection as a vacuum manifold, and a Landau--Ginzburg model. This framework proved to be a useful asset to formulate a physical proof of mirror symmetry in the context of toric varieties (see for example \cite{horivafaetal}). Furthermore, the choice of a non abelian gauge group allows the coexistence of multiple geometric phases (i.e. with a smooth Calabi--Yau variety as vacuum manifold) which are expected to be derived equivalent \cite{herbsthoripage}. This provides a method to construct candidate pairs of derived equivalent and non isomorphic Calabi--Yau varieties, which are in general hard to find. On the other hand, Fano varieties can be interesting for the purpose of understanding a possible generalization of mirror symmetry: while the existence of rigid Calabi--Yau varieties challenges the notion of mirror family, it has been proposed that the whole mirror construction should be understood at the level of \emph{supermanifolds} \cite{sethi, garavusokreuzernoll} where Fano varieties appear as bosonic components. In this context, a better understanding of the derived category of Fano manifolds can shed some light on the super-Calabi--Yau formulation of homological mirror symmetry conjectures.

\subsection*{Acknowledgements}
We would like to thank Sergio Cacciatori, Jacopo Gandini, Riccardo Moschetti and Fabio Tanturri for sharing valuable insight. 
 {We would also like to express our gratitude to Akihiro Kanemitsu for helpful comments to the first version of this paper, and for pointing out the reference \cite{pasquier}. We thank the anonymous referee for pointing out a flaw in the proof of the main theorem, and for the careful and thorough reading.} EF and GM are members of the INDAM-GNSAGA. EF, GM and MR are partially supported by PRIN2017 "2017YRA3LK", GM and MR are partially supported by PRIN2020 "2020KKWT53". MK is supported by the project Narodowe Centrum Nauki 2018/31/B/ST1/02857.

\section{Construction}

We work over the field of complex numbers. In the following, we will call \emph{projective bundle} of rank $r$ (or $\PP^{r-1}$-bundle) on a smooth projective variety $X$ the projectivization $\PP(E)$ of a vector bundle $E$ of rank $r$ over $X$, together with a map $h:\PP(E)\arw X$ which itself will be called \emph{projective bundle structure}.  {With the expression  $V_d$ we will denote a vector space of dimension $d$, and we will use the notation $V_d[-m]$ for a complex equal to zero in every degree except for degree $m$, where it is isomorphic to $V_d$.}

\subsection{Generalized homogeneous roofs}
In \cite{kanemitsu}, the notion of roof has been introduced. 
Instead of the original definition, we will recall the following characterization which is the most suitable for the purpose of this paper:

\begin{proposition}\cite[Proposition 1.5]{kanemitsu}
    Let $X$ be a smooth projective Fano variety of Picard rank two, such that its extremal contractions are $\PP^{r-1}$-fibrations. Then the following are equivalent:
    \begin{enumerate}
        \item $X$ is a roof of $\PP^{r-1}$-bundles;
        \item the index of $X$ is $r$;
        \item there exists a line bundle $\Lc$ such that $\Lc$ restricts to $\Oc(1)$ on the fibers of both the extremal contractions.
    \end{enumerate}
\end{proposition}

Clearly, if we allow the two contractions to be projective space fibrations with different relative dimensions, the only statement which is sucsceptible to generalization is the third one. Therefore we introduce the following definition:

\begin{definition}\label{def:gen_roof}
    A \emph{generalized roof} is a smooth projective Fano variety $X$ of Picard number two such that:
    \begin{enumerate}
        \item $X$ admits two projective bundle structures;
        \item there exists a line bundle $\Lc$ on $X$ such that $\Lc$ restricts to $\Oc(1)$ on both the projective bundle structures.
    \end{enumerate}
\end{definition}


Let us now consider the situation when a generalized roof is a quotient $G/P$ of a simply connected semisimple Lie group by a parabolic subgroup. We give the following definition:

\begin{definition}\label{def:gen_hom_roof}
    A \emph{generalized homogeneous roof} is a quotient $G/P$ of a simply connected, semisimple Lie group by a parabolic subgroup, which has Picard number two and admits two projective bundle structures.
\end{definition}

By \cite[Proposition 2.6]{campanapeternellsurvey}, the only possible projective bundle structures are morphisms to $G$-Grassmannians, hence the data of a generalized homogeneous roof can be represented by a diagram like the following:

\begin{equation}
    \begin{tikzcd}
        & G/P\ar[swap]{dl}{h_1}\ar{dr}{h_2}& \\
        G/P_1 && G/P_2
    \end{tikzcd}
\end{equation}

\begin{remark}
    Note that a generalized homogeneous roof is always a generalized roof: condition $(2)$ in Definition \ref{def:gen_roof} is satisfied for every quotient $G/P$ of a simply connected semisimple lie group by a parabolic subgroup such that $G/P$ has Picard number two and it has two projective bundle structures, by fixing $\Lc:=\Oc(1,1) = h_1^*\Oc(1)\otimes h_2^*\Oc(1)$. In fact, one has $G/P\simeq\PP(h_{i*}\Oc(1,1))$ for $i\in\{1;2\}$, which is equivalent to $(2)$.
\end{remark}


We can extend the classification \cite[Section 5.2.1]{kanemitsu} to generalized homogeneous roofs by means of the same method which leads to the following result, where we adopt the same kind of nomenclature used by Kanemitsu (listed in the column ``type''):

\begin{proposition}
    Let $G/P$ be a generalized homogeneous roof. Then it falls in one of the following cases:
    \begin{table}[H]
        \centering
        \small
        \begin{tabular}{ c|c|c|c|c|c}
             & $G$ & type& $G/P$ & $G/P_1$ & $G/P_2$\\ 
             \hline
             &  &   &   &   &   \\
             &  $SL(n)\times SL(m)$  & $A_{n-1}\times A_{m-1}$ & $\PP^{n-1}\times\PP^{m-1}$ & $\PP^{n-1}$ & $\PP^{m-1}$\\ 
             &    $SL(n)$  & $A_{n-1}^M$ & $F(1, n-1, n)$ & $\PP^{n-1}$ & $\PP^{n-1}$\\
             &    $SL(n)$  & $A_{k, n-1}^G$ & $F(k, k+1, n)$ & $G(k, n)$ & $G(k+1, n)$\\
             &     {$SO(7)$}  &  {$B_3^*$} &  {$OF(1, 3, 7)$} &  {$OG(1, 7)$} &  {$OG(3, 7)$}\\
             
             &    $SO(n)$  & $B_{\frac{n-1}{2}}$ & $OF(\frac{n-3}{2},\frac{n-1}{2},n)$ & $OG(\frac{n-3}{2}, n)$ & $OG(\frac{n-1}{2}, n)$\\
             &    $Sp(n)$ ($n$ even)   & $C_{k, n/2-1}$ & $IF(k-1, k, n)$ & $IG(k-1, n)$ & $IG(k, n)$\\
             &    $Spin(n)$ ($n$ even)  & $D_n$ & $OG(\frac{n}{2}-1, n)$ & $OG(\frac{n}{2}, n)^+$ & $OG(\frac{n}{2}, n)^-$\\
             &    $F_4$ & $F_4$ &$F_4/P^{2,3}$ & $F_4/P^2$ & $F_4/P^3$\\
             &    $G_2$ & $G_2$ & $G_2/P^{1,2}$ & $G_2/P^1$ & $G_2/P^2$\\
        \end{tabular}
        \caption{Generalized homogeneous roofs}\label{tab_rooflist}
    \end{table}
\end{proposition}

\begin{proof}
    A rational homogeneous variety $G/P$ is uniquely determined by the data of a marked Dynkin diagram $(D_G, I_P)$, where $D_G$ is the Dynkin diagram of $G$ and $I_P$ is the set of nodes corresponding to $P$. In fact, there is a one-to-one correspondence between parabolic subgroups of $G$ and subsets of the set $I_B$ of nodes of $D_G$ (see for instance \cite[Section 2.2]{campanapeternellsurvey}). Moreover, an inclusion of parabolic subgroups $P'\subset P$ corresponds to a contraction $G/P\arw G/P'$, and to an inclusion of sets $I_{P'}\subset I_P$. The contraction is locally trivial, and the fiber is isomorphic to the rational homogeneous variety identified by the pair $(D_G\setminus I_{P'}, I_P\setminus I_{P'})$, where $D_G\setminus I_{P'}$ is given by erasing from $D_G$ the nodes $I_{P'}$ and taking the connected component containing $I_P\setminus I_{P'}$ \cite[Recipe 2.4.1]{bastoneastwood}.
    
    Let us now consider a roof $G/P$ with projections $h_i:G/P\arw G/P_i$ for $i\in\{1;2\}$. Call $(D_G, I_{P})$, $(D_G, I_{P_1})$ and $(D_G, I_{P_2})$ the marked Dynkin diagrams corresponding to respectively $G/P$, $G/P_1$ and $G/P_2$. By the above, the fibers of $h_i$ are all isomorphic to the rational homogeneous variety defined by the marked Dynkin diagram $(D_G\setminus I_{P_i}, I_P\setminus I_{P_i})$. Hence, we can check case by case the data $(D_G, I_P, I_{P_1}, I_{P_2})$ such that the rational homogeneous varieties associated to $(D_G\setminus I_{P_1}, I_P\setminus I_{P_1})$ and $(D_G\setminus I_{P_2}, I_P\setminus I_{P_2})$ are isomorphic to projective spaces, i.e. falling in one of the following two cases:
    \begin{equation*}
        \huge
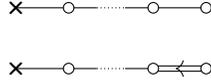
        \begin{split}
            \dynkin[edge length=20pt, root radius = 2pt]{A}{xo.oo}\\
            \dynkin[edge length=20pt, root radius = 2pt]{C}{xo.oo}
        \end{split}
    \end{equation*}
    where we crossed out the nodes correspond to $I_P\setminus I_{P_i}$. This explicit analysis eventually leads to Table \ref{tab_rooflist}, thus proving our claim.
\end{proof}

\begin{remark}
    The same list of Table \ref{tab_rooflist} has been obtained by Pasquier \cite[Theorem 1.7]{pasquier} in the context of horospherical varieties of Picard rank one.  
\end{remark}

\begin{remark}
    Note that the new cases listed in Table \ref{tab_rooflist} with respect to \cite[Section 5.2.1]{kanemitsu} are  $A_{m-1}\times A_{n-1}$ for $m\neq n$, $A^G_{k, n-1}$ for $k\neq \frac{n-1}{2}$, $C_{k,n/2-1}$ for $n=\frac{3k-2}{2}$, the whole type $B_{\frac{n-1}{2}}$ construction and the isolated case we denoted as $B^*_3$.
\end{remark}

In light of the definition of Calabi--Yau pairs associated to a roof \cite[Definition 2.5]{kr2}, we formulate the following:

\begin{definition}\label{def:pair}
    A \emph{pair} $(Y_1, Y_2)$ associated to $G/P$ is the data of two smooth varieties $Y_1\subset G/P_1$ and $Y_2\subset G/P_2$ such that there is a general section $S\in H^0(G/P, \Oc(1,1))$ with $Z(h_{i*}S) = Y_i$ for $i\in\{1;2\}$.
\end{definition}

\section{Hodge theoretical analysis for type \texorpdfstring{$A_{k, n-1}$}{}: Fano vs General type scenario}

Let us consider a vector space $V\simeq\CC^n$. For any integer $k<n$, we call $G(k, V)$ the Grassmannian of $k$-linear subspaces of $V$. 
More generally, given any $r$-tuple $k_1<\dots<k_r<n$, we denote by $F(k_1,\dots, k_r,V)$ the flag variety parametrizing the $r$-tuples $V_1\subset\cdots\subset V_r\subset V$ of flags of linear subspaces of $V$. Every flag variety is Fano of Picard rank $r$, with anti-canonical class $\omega_F^{\vee} \cong \of(k_2, k_3-k_1, \dots, n-k_{r-1})$, with respect to the basis given by the pullbacks of the Pl\"ucker classes. In particular, for every $k<n$, $G(k, V)$ is a Fano $kn-k^2$-fold of index $n$.

On $G(k, V)$ we have the tautological sequence

\begin{equation}\label{eq:tautologicalseq}
    0\arw \Uc\arw V\otimes\Oc\arw \Qc \arw 0
\end{equation}

where $\Uc$ is the rank $k$ tautological bundle, with determinant $\Oc(-1)$, which reduces to the dual Euler sequence for $k=1$.

The generalized roof of type $A_{k,n-1}$ is the flag variety $F(k, k+1, V)$, which has two projections $h_1$ and $h_2$, respectively onto $G(k, V)$ and $G(k+1, V)$. Given any point $x\in G(k, V)$, and the associated $k$-linear space $W_x\subset V$, one has $h_1^{-1}(x) = \{W_y\in G(k+1, V) ~ : ~ W_x\subset W_y\}\simeq G(1, V/W_x)$. Similarly one has $h_2^{-1}(y) = \{x\in G(k, V) ~ : ~ W_x\subset W_y\}\simeq G(k, W_y)$. In particular, $h_1$ is a $\PP^{n-k-1}$ bundle, and $h_2$ a $\PP^{k}$ bundle and one has $h_{1*}\Oc(1,1) = \Qc^\vee(2)$ and $h_{2*}\Oc(1,1) = \Uc(2)$.

We obtain the diagram:

$$
\begin{tikzcd}[column sep = tiny]
& \Fl:=F(k, k+1, V)\ar{dl}{h_1}\ar[swap]{dr}{h_2} & \\
G(k, V) & &G(k+1, V)
\end{tikzcd}
$$

Consider now the zero locus $M$ of a section $S \in H^0(\Fl, \of(1,1)) \cong \Sigma_{2^{k-1},1} V$, namely the cokernel of the comultiplication map  $\W^{k-1}V\otimes \W^{k+2}V \to \W^{k}V\otimes \W^{k+1}V $.
Moreover, the pushforwards of this section with respect to the surjections $h_1$ and $h_2$ will give sections 
\begin{align*}
    S_1 &:= h_{1*}S\in H^{0}( G (k, V), \Qc^{\vee}(2));\\
     S_2& := h_{2*}S\in H^{0}( G (k+1, V), \mU(2)).
\end{align*}

 For $i=1,2$, denote by $Y_i$  the zero locus of the section $S_i$ on the respective Grassmannian and by $\overline{h_i}$ the restriction of $h_i$ to $M$. We have the following refined diagram:

\begin{equation}\label{diagramflag}
    \begin{tikzcd}
    & M \ar{dl}{\overline{h_1}}\ar[swap]{dr}{\overline{h_2}} & \\
    Y_1 \subset  G (k, V) & & G (k+1, V) \supset Y_2
\end{tikzcd}
\end{equation}

The fibers of $\overline{h_1}$ and $\overline{h_2}$ are, respectively, generically a $\PP^{n-k-2}$ and a $\PP^{k-1}$ bundle. The locus where the dimension jumps, and the fiber coincides with the whole $\PP^{n-k-1}$ (resp. $\PP^k$) corresponds to the zero loci of $\sigma_1$ (resp. $\sigma_2$), that is $Y_1$ and $Y_2$. We have the following equality in the Grothendieck ring.
\begin{lemma}\label{leq} Let $Y_1$ and $Y_2$ be as above. Then: $$
[\LL^k][Y_2]-[\LL^{n-k-1}][Y_1]+ [ G (k+1, V)][\PP^{k-1}]-[ G (k,V)][\PP^{n-k-2}]=0.
$$
\end{lemma}
\begin{proof}
It follows from the above description of $M$ as a stratified projective bundle with respect to both projections, and from the properties of the Grothendieck ring, e.g. $[\PP^k]-[\PP^{k-1}] = [\LL^k]$.
\end{proof}

\begin{proposition}
    
\label{cor:picard}
$Y_1$ and $Y_2$ (when of dimension $\geq 3$) have Picard rank $\rho=1$.
\end{proposition}

\begin{proof}
We show the proof for $Y_2$. From the description of $M$ as stratified projective bundle, we know that the following equivalence holds in the Grothendieck ring:
\begin{equation}\label{eq:projbundlestruct} [\LL^k][Y_2]+[\Gr(k+1, n)][\PP^{k-1}]= [M].
\end{equation}
In order to determine $\rho$, we can solve the equation and compare the appropriate degree. Namely, if we denote $\beta_i:= b_{2i}(\Gr(k+1, n))$ (i.e. the $2i$ Betti number of the Grassmannian), and we equate the components of degree $k+1$ in \ref{eq:projbundlestruct}:
\begin{equation} \label{eqn:groth}
b_2(Y_2) +\sum_{i=2}^{k+1} \beta_i = \gamma.
\end{equation}
But $M$ has dimension $(k+1)(n-k-1)+k-1$, which we need to be greater than $2(k+1)$. The equation simplifies to
$(k+1)(n-k-2)>2$. This implies that either $k=1$ and $n>4$, or $k\geq 2$ and any $n$.
The only case excluded is therefore for the roof $F(1,2,4)$, where $Y_2$ is in fact a del Pezzo surface of degree 5. Notice that in all other cases, $Y_2$ has dimension greater than 3, as requested.

We can then apply Lefschetz theorem on hyperplane section, and it follows that $\gamma= b_{2(k+1)}(F(k,k+1, n))$. On the other hand we know that $F(k, k+1, n)$ is a $\PP^k$-bundle over $\Gr(k+1, n)$. Hence, by K\"unneth formula, it follows that $b_{2(k+1)}(F(k,k+1, n))= \sum_{j=1}^{k+1} \beta_i$. Substituting in \ref{eqn:groth} we get $b_2(Y_2)= \beta_1=1$. This implies in particular that $\rho(Y_2)=1$.
The proof for $Y_1$ is very similar, and we will therefore omit it.
\end{proof}

We can compute more invariants of $Y_1$ and $Y_2$, for example as follows.\begin{lemma} For general $S$, $Y_1$ has dimension $d_1:=kn-k^2-n+k$ and canonical class $\omega_{Y_1}\cong \of_{Y_1}(n-2k-1)$. $Y_2$ has dimension $d_2:=kn + n - k^2 - 3k - 2$ and canonical class $\omega_{Y_2}\cong \of_{Y_2}(-n + 2k +1)$.
\end{lemma}
\begin{proof}
The canonical class can be easily computed by adjunction, since the two varieties can be expressed as zero loci of homogeneous vector bundles. 
\end{proof}

\begin{remark}\label{rem:canonicalclass}
For $n = 2k+1$, the varieties $Y_1$ and $Y_2$ are Calabi--Yau of the same dimension $k^2-1$. However, for $l>k+1$, $Y_1$ is of general type and $Y_2$ is Fano, with dim $ Y_1 <$ dim $Y_2$.
\end{remark}

We want to understand the Hodge-theoretical relation between $Y_1$ and $Y_2$. The upshot is that the Hodge structure of $Y_2$ will be essentially the same as the one of $Y_1$, up to some class from the ambient Grassmannian(s). In order to do this, we introduce the following definition:

\begin{definition} \label{def:v-cohom}
    Let $X$ be the $n$-dimensional smooth zero locus of a section of a globally generated vector bundle $\Fc$ on $G$. If $\iota: X \hookrightarrow G$ denotes the inclusion morphism we will denote by 
   
    \[ H^n_{f}(X, \QQ):= \mathrm{Im}(\iota^*\colon H^n(G, \QQ) \to H^n(X, \QQ)). \]
    We will define the \emph{v-cohomology} $H^n_{v}(X, \QQ)$ as the orthogonal to  $H^n_{f}(X, \QQ)$ with respect to the cup product. Namely we have \[ H^n(X, \QQ) \cong H^n_{f}(X, \QQ) \oplus H^n_{v}(X, \QQ).\]
    
\end{definition}

The terminology \emph{v-cohomology} is supposed to echo both the \emph{vanishing cohomology} and the \emph{variable cohomology}, which are defined for smooth complete intersections, see \cite{peters}, \cite[2.3.3]{voisin} (similarly, its orthogonal complement is supposed to echo the \emph{fixed} part of the cohomology). It is worth noting that if $\mF$ splits as a sum of ample line bundles, then $\iota^*$ is in fact injective as a consequence of the Lefschetz hyperplane theorem, and $H^n_v$ coincides with the variable cohomology. Such property holds as well if $\mF|_X$ is ample, see \cite[5.2]{ottem}. 

We will show the following theorem.

\begin{theorem} \label{thm:hodge}
Let $Y_1$ and $Y_2$ be as in Diagram \ref{diagramflag}. Then we have the following isomorphisms of rational Hodge structures:
\begin{align*}  
H^{d_1}(Y_1, \QQ) & \cong H^{d_1}( G (k, V), \QQ) \oplus H^{d_1}_v(Y_1, \QQ), \\  
H^{d_2}(Y_2, \QQ) & \cong  H^{d_2}( G (k+1, V), \QQ) \oplus H^{d_2}_v(Y_2, \QQ). 
\end{align*} 

 Moreover
\[H^{d_1}_v(Y_1, \QQ) \cong H^{d_2}_v(Y_2, \QQ).\]
\end{theorem}

\begin{proof}
Observe that the first part of the theorem is equivalent to the injectivity of $\iota_i^*$, where $\iota_i$ is the embedding of $Y_i$ in the corresponding Grassmannian, for $i\in\{1;2\}$. We start by applying \cite[Prop. 48]{bfm} to both sides of Diagram \ref{diagramflag}, where we do not write the shifts in cohomology for the sake of clarity, and we always take the cohomology with $\QQ$-coefficients. We obtain
\begin{align} \label{eq:eqdiag1}
    H^{d_M}(M) & \cong \bigoplus_{i=0}^{k-1} H^{d_M-2i}( G (k+1, V)) \oplus H^{d_2}(Y_2);\\
     H^{d_M}(M) & \cong  \bigoplus_{i=0}^{n-k-2} H^{d_M-2i}( G (k, V)) \oplus H^{d_1}(Y_1).
\end{align}
By K\"unneth formula we compute $H^{d_M}(\Fl)$ in two different ways as 
\begin{equation}\label{eq:kunneth}
H^{d_M}(\Fl) \cong \bigoplus_{i=0}^k H^{d_M-2i}( G (k+1, V)) \cong \bigoplus_{i=0}^{n-k-1} H^{d_M-2i}( G (k, V)).
\end{equation}
On the other hand, $M \subset \Fl$ is cut by an ample divisor, hence by the Lefschetz hyperplane section theorem we can combine Equations \ref{eq:kunneth} and \ref{eq:eqdiag1} to get injective maps
\begin{align} \label{eq:injective}
\bigoplus_{i=0}^{k} H^{d_M-2i}( G (k+1, V)) & \hookrightarrow \bigoplus_{i=0}^{k-1} H^{d_M-2i}( G (k+1, V)) \oplus H^{d_2}(Y_2); \\
    \bigoplus_{i=0}^{n-k-1} H^{d_M-2i}( G (k, V)) & \hookrightarrow  \bigoplus_{i=0}^{n-k-2} H^{d_M-2i}( G (k, V)) \oplus H^{d_1}(Y_1).
\end{align}
Since $d_M-2k=d_2$ and $d_M-2n+2k+2=d_1$, this implies that $\iota_2^*: H^{d_2}( G (k+1, V))\arw H^{d_2}(Y_2)$ is injective, and $\iota_1^*$ as well, thus proving the first part of the statement.

We then write $H^{d_i}(Y_i) \cong H^{d_i}( G _i) \oplus H^{d_i}_v(Y_i)$. From \ref{eq:eqdiag1} we get 
\[ B \oplus H^{d_1}_v(Y_1) \cong A \oplus H^{d_2}_v(Y_2), \]
where \[A= \bigoplus_{i=0}^{k-1} H^{d_M-2i}( G (k+1, V))  \oplus H^{d_M-2k}( G (k+1, V)) \] and \[B= \bigoplus_{i=0}^{n-k-2} H^{d_M-2i}( G (k, V)) \oplus H^{d_M-2l+2}( G (k, V)).\]
Thanks to \ref{eq:kunneth} we have $A\cong B$. The result follows.
\end{proof}
\begin{remark}
    The above proof, for $d_M$ odd, works also over $\mathbb{Z}$ as the decompositions in \eqref{eq:eqdiag1} are trivial: this generalizes the result \cite[Proposition 3.4]{kr2} which worked only for the Calabi--Yau case. However, when $d_M$ is even, the main difficulty is that the orthogonal decomposition in Definition \ref{def:v-cohom} works only with rational coefficients: in fact, as pointed out in \cite[2.4]{bs}, such decomposition in general does not hold over $\ZZ$ as the sum of the two sublattices can be not saturated. However, many of the tools used in the proof of \ref{thm:hodge} works over $\ZZ$ as well. 
\end{remark}
\begin{corollary}
Both $Y_1$ and $Y_2$ have only middle Hodge cohomology (either horizontal or vertical). In other words, for $p+q \neq d_i$, $H^{p,q}(Y_i)=0$ unless $p=q$.
\end{corollary}
\begin{proof}
This holds true for $M$ thanks to the Lefschetz hyperplane theorem. The result follows from \cite[Prop. 48]{bfm}, applied as in \ref{eq:eqdiag1} to all the other degrees in cohomology.
\end{proof}
\begin{corollary} \label{cor:coho}
The middle cohomology of $Y_2$ has Hodge co-level $n-2k-1$. In other words, for $p+q= d_2$, $h^{p,q}(Y_2) = 0$ for $p<n-2k-1$, and $h^{n-2k-1,k(n-k-1)-1}(Y_2) \neq 0$ . 
\end{corollary}

\subsection{Examples}
Computing Hodge numbers of a subvariety $X$ of a Grassmannian $G (k,n)$ cut by a section of a homogeneous, completely reducible, globally generated vector bundle $\Fc$ is (almost) an algorithmic procedure.

We briefly recall the strategy. Set rank$(\mF)=r$. Also, denote by $G:= G (k,n)$. For each $j \in \mathbb{N}$, we have the $j$-th exterior power of the conormal sequence
\begin{equation}
\label{wedgeKConormal}
0\rightarrow
\Sym^j \mF^\vee|_X \rightarrow
(\Sym^{j-1} \mF^\vee \otimes \Omega_G)|_X \rightarrow
\dotso \rightarrow
(\Sym^{j-k} \mF^\vee \otimes \Omega^k_G)|_X \rightarrow
\dotso \rightarrow
\Omega^j_G|_X \rightarrow
\Omega^j_X \rightarrow
0.
\end{equation}
Since we want to determine $h^i(\Omega^j_X)$, we can compute the dimensions of the cohomology groups of all the other terms in \eqref{wedgeKConormal}, split it into short exact sequences and use the induced long exact sequences in cohomology to get the result.

Each term $(\Sym^{j-k} \mF^\vee \otimes \Omega^k_G)|_X$ is in turn resolved by a Koszul complex
\begin{equation*}
0\rightarrow
\W^r \mF^\vee \otimes \Sym^{j-k} \mF^\vee \otimes \Omega^k_G \rightarrow
\dotso \rightarrow
\mF^\vee \otimes \Sym^{j-k} \mF^\vee \otimes \Omega^k_G \rightarrow
\Sym^{j-k} \mF^\vee \otimes \Omega^k_G,
\end{equation*}
so that we are led to compute the cohomology groups of the terms above. Since by hypothesis $\mF$ is completely reducible, then those terms are completely reducible as well: a decomposition can be found via suitable plethysms. The cohomology groups can be then obtained via the usual Borel--Weil--Bott Theorem. Most of these computations can be sped up by using computer algebra systems such as Macaulay2 \cite{M2}.

To see explicit examples of these computations we refer to \cite[3.3]{dft} or \cite[3.9.1]{eg2}.

\subsubsection{The Calabi--Yau case}

As we noted in Remark \ref{rem:canonicalclass}, for the choice $n = 2k+1$ the two zero loci are smooth Calabi--Yau varieties. This case has been investigated for $k = 2$ in \cite{kr} where the pair has been proven to be derived equivalent, $\LL$-equivalent and non birational. For $k>2$ the derived equivalence is very difficult to settle. On the contrary, $\LL$-equivalence is simple to determine for every $k$, as it follows from \cite[Remark 2.8]{kr2} and the fact that $G(k, 2k+1)\simeq G(k+1, 2k+1)$.

Notice that Theorem \ref{thm:hodge} gives an isomorphism between $H^{k^2-1}(Y_1, \QQ) \cong H^{k^2-1}(Y_2, \QQ)$. It is therefore natural to conjecture the following.
\begin{conjecture}
     For $Y_1$ and $Y_2$ Calabi-Yau $k^2-1$ as above, we have that $H^{k^2-1}(Y_1, \ZZ) \cong H^{k^2-1}(Y_2, \ZZ)$, but generically $Y_1 \not\cong Y_2$.
\end{conjecture}
We remark that the dimension and the Hodge numbers grow quickly, although only in the central part, thanks to \ref{cor:coho}. To give an example, for the $F(3,4,7)$ case, the two Calabi-Yau 8-folds whose Hodge numbers are collected in the following lemma.
\begin{lemma}
The non-zero Hodge numbers of the above 8-folds are:
\begin{itemize}
     \item (for $p+q \neq 8$): $h^{0,0}=h^{1,1}=1$, $h^{2,2}=2$, $h^{3,3}=3$.
     \item (for $p+q=8$): $h^{8,0}=1$, $h^{7,1}=735$, $h^{6,2}=41161$, $h^{5,3}=395626$, $h^{4,4}=825751$.
\end{itemize}
where we avoided the repetitions induced by the obvious Hodge symmetries.
\end{lemma}

\subsubsection{21 points and a Fano 6-fold}
Let us consider now an example for which we have both a derived categorical and a Hodge theoretical statement. Consider in fact the generalized homogeneous roof $F(1,2,6)$ and the associated pair $(Y_1, Y_2)$ in the sense of Definition \ref{def:pair}. Via the Euler sequence on $\PP^5$, one can immediately compute that $Y_1$ consists of 21 points. On the other hand $Y_2$ is a Fano 6-fold and its Hodge numbers, which can be already found in \cite[Appendix C]{eg2}, are as follows:

\begin{center}
{\small
\[\begin{matrix}
&&&&&&&&1 &&&&&&&&\\
&&&&&&&0&&0&&&&&&&\\
&&&&&&0 &&1&&0&&&&&&\\
 &&&&& 0 && 0 && 0 &&0&&& \\
&&&&0 &&0 && 2 &&0 && 0 &&&&\\
&&&0&&0 & & 0 & &0 && 0 && 0 &&&\\
&&0 && 0 && 0 &&22 &&0 &&0 && 0&&\\
&&&0&&0 & & 0 & &0 && 0 && 0 &&&\\
&&&&0 &&0 && 2 &&0 && 0 &&&&\\
&&&&& 0 && 0 && 0 &&0&&& \\
&&&&&&0 &&1&&0&&&&&&\\
&&&&&&&0&&0&&&&&&&\\
&&&&&&&&1 &&&&&&&&
\end{matrix}\]}
\end{center}

Since $h^0(\PP^5)=1$ and $h^6( G (2,6))=2$ the two v-cohomology subspaces are both 20-dimensional, and by Theorem \ref{thm:hodge}, we have $H^0_v(Y_1) \cong H^6_v(Y_2)$. Also, by the results of the next section, we have an embedding of the derived category of $Y_1$ in the derived category of $Y_2$.

\subsubsection{A general type 4-fold and a Fano 6-fold}

The last example we consider is pair $(Y_1,Y_2)$ associated to $F(2,3,6)$: $Y_1$ is a 4-fold of general type with $\omega_{Y_1} \cong \of_{Y_1}(1)$, and that $Y_2$ is a Fano 6-fold of index one. The Hodge numbers for $Y_1$ are:

\begin{center}
{\small
\[\begin{matrix}
&&&&&&1 &&&&&&\\
&&&&&0&&0&&&&&\\
&&&&0 &&1&&0&&&&\\
&&& 0 && 0 && 0 &&0&&& \\
&&15 &&672 && 2271 &&672 && 15 &&\\
&&& 0 && 0 && 0 &&0&&& \\
&&&&0 &&1&&0&&&&\\
&&&&&0&&0&&&&&\\
&&&&&&1 &&&&&&
\end{matrix}\]}
\end{center}

On the other hand the Hodge numbers for $Y_2$ are:

\begin{center}
{\small
\[\begin{matrix}
&&&&&&&&1 &&&&&&&&\\
&&&&&&&0&&0&&&&&&&\\
&&&&&&0 &&1&&0&&&&&&\\
 &&&&& 0 && 0 && 0 &&0&&& \\
&&&&0 &&0 && 2 &&0 && 0 &&&&\\
&&&0&&0 & & 0 & &0 && 0 && 0 &&&\\
&&0 &&15 &&672 && 2272 &&672 && 15 && 0&&\\
&&&0&&0 & & 0 & &0 && 0 && 0 &&&\\
&&&&0 &&0 && 2 &&0 && 0 &&&&\\
&&&&& 0 && 0 && 0 &&0&&& \\
&&&&&&0 &&1&&0&&&&&&\\
&&&&&&&0&&0&&&&&&&\\
&&&&&&&&1 &&&&&&&&
\end{matrix}\]}
\end{center}

We can immediately check that $h^4( G (2,6))=2$ and $h^6( G (3,6))=3$: hence the two v-cohomologies subspaces have the same rank and by Theorem \ref{thm:hodge} one has:
\begin{align*}
    H^4(Y_1, \QQ) &\cong \QQ^2 \oplus H^4_v(Y_1)\\
    H^6(Y_2, \QQ) &\cong \QQ^3 \oplus H^6_v(Y_2),
\end{align*}
with $H^4_v(Y_1)\cong  H^6_v(Y_2)$.

We conjecture the categorical equivalent of this result: i.e we expect that $D^b(Y_1)$ injects in $D^b(Y_2)$. Notice that this would make $Y_2$ a Fano host for the visitor $Y_1$ in the sense of, e.g. \cite{kkll}.
 
\section{Derived embedding and window categories}

\subsection{Setup and general strategy}\label{subsection:setup}

Hereafter we will heavily rely on the notion of gauged linear sigma model (GLSM). Therefore, let us begin this section by recalling some general terminology. 

\begin{definition}\label{def:glsm}
    We call \emph{gauged linear sigma model} $(W, G, \CC^*_R, f)$ the following data:
    \begin{enumerate}
        \item A finite dimensional vector space $W$;
        \item A linear reductive group $G$ with an action on $W$;
        \item An action of $\CC^*$ on $W$ called \emph{$R$-symmetry}, traditionally denoted by $\CC^*_R$;
        \item A polynomial $f:W\arw\CC$ called \emph{superpotential}.
    \end{enumerate}
    Moreover, we require the following conditions to hold:
    \begin{enumerate}
        \item The $G$-action and the $\CC^*_R$-action commute on $W$;
        \item $f$ is $G$-invariant and $\CC^*_R$-homogeneous with positive weight.
    \end{enumerate}
\end{definition}

\begin{definition}\label{def_glsmphase}
    Let $(W, G, \CC^*_R, f)$ be a GLSM. We call \emph{phase} of the GLSM a chamber of the associated polyhedral fan (see \cite[Theorem 3.3]{halic} for details on the walls-and-chambers structure).
\end{definition}

\begin{definition}\label{def:glsm_general}
    Let $(W, G, \CC^*_R, f, I)$ be a GLSM phase, where $I$ is the associated chamber. We call \emph{critical locus} of the superpotential:
    \begin{equation}
        \operatorname{Crit}(f):=Z(df)
    \end{equation}
    where $df$ is the gradient of $f$. Moreover, we call \emph{vacuum manifold} the GIT quotient:
    \begin{equation}
        Y_I = \operatorname{Crit}(f)\git_\tau G. 
    \end{equation}
    for any $\rho_\tau\in I$.
\end{definition}

Let us fix a complex vector space $V_n$ of dimension $n$. The data of the generalized roof of type $A_{k, n-1}^G$ is related to the flag variety $F(k, k+1, V_n)$, which can be summarized by the following diagram:

\begin{equation}\label{eq:diagram12n}
    \begin{tikzcd}
        & F(k,k+1,V_n)\ar[swap]{dl}{h_{ {1}}}\ar{dr}{h_{ {2}}} & \\
        G(k,V_n) & & G(k+1, V_n)
    \end{tikzcd}
\end{equation}

Fix a general global section $S$ of $\Oc(1,1)$: it cuts a smooth hypersurface $M\subset {F(k,k+1,V_n)}$, and $h_{1*}S$ and $h_{2*}S$ define the pair $(Y_1, Y_2)$ where $Y_1 = Z(h_{1*}S)$ is of general type and $Y_2 = Z(h_{2*}S)$ is Fano when $k< n/2$ (the converse happens for $k\geq n/2$, while both $Y_1$ and $Y_2$ are equidimensional Calabi--Yau varieties for $n$ even and $k=(n-1)/2$).

The case $A_{1, n-1}^G$ corresponds to setting $k=1$, but for our purpose it will be more convenient to consider the isomorphic roof $A_{n-2, n-1}^G$ (i.e. $k = n-2)$. In this case, $Y_2$ is a set of points and $Y_1$ is a Fano $(2n-6)$-fold of index $n-3$. Our goal now is to prove the existence of a fully faithful functor $\dbcoh Y_2\xhookrightarrow{\,\,\,\,\,\,}\dbcoh Y_1$. This will be done in several steps:
\begin{enumerate}
    \item[$\circ$] For any $k$, consider the varieties $X_+:=\operatorname{tot}(\Uc^\vee(-2)\arw G(k+1, V_n))$ and $X_-:=\operatorname{tot}(\Qc(-2)\arw G(k, V_n))$, i.e. the total spaces of $ {\Uc^\vee}(-2)$ and $\Qc(-2)$ respectively on $G(k+1, V_n)$ and $G(k, V_n)$. In Section \ref{subsection:glsm}, given an appropriate vector space $W$ and affine subvarieties $Z_\pm$, we will find explicit GIT descriptions $X_\pm:=(W\setminus Z_\pm)/GL(V_{k+1})$, and this will allow us to define two birational maps $i_\pm:X_\pm\dashrightarrow \Xc:=[W/GL( {V_{k+1}})]$ which are always isomorphisms outside loci of codimension at least two.
    
    \item[$\circ$] In Section \ref{subsection:criticalloci}, for any $k$ we fix the data $(W, G, \CC^*_R, f)$ of a GLSM such that the vacuum manifolds of two different phases are isomorphic to $Y_1$ and $Y_2$. This is necessary to identify $\dbcoh Y_1$ and $\dbcoh Y_2$ with the right $B$-brane categories via Kn\"orrer periodicity \cite[Theorem 3.4]{shipman}.
    
    \item[$\circ$] Fix $k = n-2$. In Section \ref{subsection:windows} we identify a subcategory $\Wc_0\subset \dbcoh(\Xc)$ such that the functors $i_\pm^*|_{\Wc_0}:\Wc\arw\dbcoh(X_\pm)$ are fully faithful, and we define $\Wc_0$ such that $i_\pm^*|_{\Wc_0}$ is also essentially surjective. This determines a fully faithful functor $\dbcoh X_-\xhookrightarrow{\,\,\,\,\,\,}\dbcoh X_+$ which is lifted to a functor of $B$-brane categories in Section \ref{subsection:branes}. The latter yields the desired functor $\dbcoh Y_{2}\xhookrightarrow{\,\,\,\,\,\,}\dbcoh Y_{1}$ once composed with Kn\"orrer periodicity.
\end{enumerate}

\subsection{Gauged linear sigma model and variation of GIT}\label{subsection:glsm}


Let us describe how the spaces $X_+$ and $X_-$ can be related by a variation of GIT.

\begin{lemma}\label{lem:git_quotients}
    Let us consider the vector space $W = \Hom(V_{k+1},V_n)\oplus \Hom(V_{k+1}, \Sym^2(\wedge^{k+1} V_{k+1}))$, with $GL(V_{k+1})$ acting on $V_{k+1}$ via the fundamental representation and trivially on $V_n$.
    There exist two GIT quotients of $W$ with respect to $GL(V_{k+1})$:
    \begin{equation}
        \begin{split}
            X_+ &=  W \git_+ GL(V_{k+1})\simeq \operatorname{tot}(\Uc^\vee(-2)\arw G(k+1, V_n))\\
            X_- &=  W \git_- GL(V_{k+1})\simeq \operatorname{tot}(\Qc(-2)\arw G(k, V_n)).
        \end{split}
    \end{equation}

\end{lemma}

\begin{proof}
 By assumption, one has the following $GL(V_{k+1})$-action on $W$:

\begin{equation}\label{eq:gl2action}
    \begin{tikzcd}[row sep=tiny, /tikz/column 1/.append style={anchor=base east} ,/tikz/column 3/.append style={anchor=base west}]
        GL(V_{k+1})\times W \ar{r} & W \\
        g, (B,\omega) \ar[maps to]{r} & (Bg^{-1}, \omega g^{-1}\det g^2 )
    \end{tikzcd}
\end{equation}

where $B$ is a $n\times k$ matrix and $\omega$ is a row vector in $V^\vee_{k+1}$, i.e. a linear map from $V_{k+1}$ to the one dimensional space $\wedge^{k+1} V_{k+1}$. Given a character $\rho:GL(V_{k+1})\arw\CC^*$ we have the $GL(V_{k+1})$-\emph{semistable locus}:

\begin{equation}\label{eq:stabcondition}
        V^{ss}_\rho=\{(B,\omega)\in W : \{0\}\times W\,\,\,\cap\,\,\,\overline{\{(\rho(g)^{-1}, g.(B, \omega))|g\in GL(V_{k+1})\}} =\emptyset\}
    \end{equation}
    
and the \emph{unstable locus} will be simply the set $Z_\rho:= V\setminus V^{ss}_\rho$. Consider now the family of characters:

\begin{equation*}
    \begin{tikzcd}[row sep=tiny, /tikz/column 1/.append style={anchor=base east} ,/tikz/column 3/.append style={anchor=base west}]
        GL(V_{k+1}) \ar{r}{\rho_\tau} &  \CC^* \\
        g \ar[maps to]{r} & \det(g)^{-\tau}.
    \end{tikzcd}
\end{equation*}

We can describe the unstable locus for the chamber $\tau>0$ as the set $Z_+$ of pairs $(B, \omega)\in W$ such that there exists a one-parameter subgroup $\{g_t\}\subset GL(V_{k+1})$ fulfilling the following conditions:

\begin{enumerate}
    \item\label{cond:convergencepositive1} The expression $\det(g_t)$ converges to zero for $t\arw 0$ (so that $\rho^{-1}_\tau(g_t)$ converges to zero)
    \item\label{cond:convergencepositive2} The expression $g_t.(B, \omega)$ has a limit in $V$ for $t\arw 0$.
\end{enumerate}

By applying such conditions to the action \ref{eq:gl2action} we find $Z_+ = \{(B,\omega)\in W :\operatorname{rk}(B) < k+1\}$, and then the corresponding GIT quotient with respect to the chamber $\tau>0$ is:

$$
W\git_+ GL(V_{k+1}) = (W\setminus Z_+)/GL(V_{k+1}) = X_+.
$$

Let us now construct the GIT quotient with respect to the chamber $\tau<0$: here the analysis is slightly more involved. For $l = (l_1, \dots, l_{k+1})\in\ZZ^{k+1}$ we consider the following one-parameter subgroups:

\begin{equation}
    \begin{tikzcd}[row sep=tiny, /tikz/column 1/.append style={anchor=base east} ,/tikz/column 3/.append style={anchor=base west}]
        t \ar[maps to]{r} & g_l(t)=\operatorname{diag}(\{t^{l_1},\dots, t^{l_{k+1}}\})
    \end{tikzcd}
\end{equation}

Here the unstable locus is given by pairs $(B, \omega)$ such that there exists a $g_l$ fulfilling the following conditions:
\begin{enumerate}
    \item\label{cond:convergencenegative1} $\det(g_l)^{-1}$ converges to zero for $t\arw 0;$ 
    \item\label{cond:convergencenegative2} $g_l .(B, \omega)$ has a limit in $V$ for $t\arw 0.$
\end{enumerate}
The only possibility for a pair $(B, \omega)$ to admit such conditions is that either $\omega = 0$ or there exist a nontrivial intersection $\ker(\omega)\cap\ker(B)$. Hence we find:
$$
W\git_- GL(V_{k+1}) = \{(B, \omega)\in W : \omega\neq 0, \ker(\omega)\cap\ker(B) = \{0 \}\}/GL(V_{k+1}).
$$
For every $\omega\neq 0$, there always exists an element $g_\omega\in GL(V_{k+1})$ such that $g_\omega.\omega = (1,0, \dots, 0)$. Hence, we can rewrite the quotient as

\begin{equation}\label{eq:git_quotient-}
    W\git_- GL(V_{k+1}) = \{(B, \omega)\in W : \omega  =(1,0,\dots,0), \ker(\omega)\cap\ker(B) = \{0 \}\}/\operatorname{Stab}
\end{equation}

with the stabilizer 

\begin{equation}
\label{eq:stab}
    \operatorname{Stab} = \left\{\left(\begin{matrix}
    \det h^{-2} & 0 \\
    \kappa & h
    \end{matrix}\right)\in GL(V_{k+1}) : \kappa\in \CC^k, h\in GL(V_k)\right \}
\end{equation}

Let us write $B = (v|C)$, i.e. rename by $v$ the first column of $B$ and by $C$ the rest of the matrix.: we see that the kernel condition imposes that $C$ must have maximal rank. Hence, once defined $\wt W := \{(v, C)\in V_n\oplus\Hom(V_k, V_n) : \rk C = k\}$,  Equation \ref{eq:git_quotient-} can be rewritten in the following way:

\begin{equation}\label{eq:git_quotient-_stab}
    W\git_- GL(V_{k+1}) = \wt W /\operatorname{Stab}
\end{equation}

where the $\operatorname{Stab}$-action is:

\begin{equation}
    (g, v , C)\longmapsto (v\det h^{2}+C\kappa, C h^{-1}).
\end{equation}

Observe that $v$, up to the action of an element of $\operatorname{Stab}$, lies in the orthogonal complement of the $k$-dimensional subspace of $V_n$ spanned by the columns of $C$: this identifies the variety of Equation \ref{eq:git_quotient-_stab} with the total space of $\Qc(-2)$ on $G(k, V_n)$.

\end{proof}

In this way we constructed two open embeddings $i_\pm:X_\pm\xhookrightarrow{\,\,\,\,\,\,\,}\Xc$. The varation of GIT we described can be extended to vector bundles described by $GL(V_{k+1})$-representations.
In fact, if we consider a representation $\Gamma:GL(V_{k+1})\arw\operatorname{End}(V_\Gamma)$, one has the following vector bundles:

\begin{equation}\label{eq:gitbundles}
    \begin{split}
        \Ec &= \left(W \oplus V_{\Gamma}\right) / GL(V_{k+1})\\
        i_\pm^* \Ec &= \left(W\setminus Z_\pm \times V_{\Gamma}\right) / GL(V_{k+1}).
    \end{split}
\end{equation}

\subsection{The critical loci}\label{subsection:criticalloci}
Let us now describe how the variation of GIT not only links the two total spaces $X_+$ and $X_-$, but also $Y_2$ and $Y_1$ as GIT quotients of the critical locus of a suitable superpotential, where $Y_2:= Y_+$ and $Y_1 := Y_-$ in the language of in Definition \ref{def:glsm_general}. To this purpose, we need a very explicit description of a general section $S\in H^0(F, \Oc(1,1))$ and its two pushforwards.\\
\\
We will denote points in $G(k+1, V_n)$ as classes $[B]$ where $B\in \Hom(V_{k+1}, V_n)$ of maximal rank, with respect to the equivalence relation given by the fundamental $GL(V_{k+1})$-action on  $V_{k+1}$ Similarly, points in $G(k, V_n)$ will be represented as classes $[A]$ with respect to an analogous $GL(V_k)$-action. Consequently, points in $F(k, k+1, V_n)$ will be pairs $([A], [B])$ with the additional condition of the image of $A$ as a linear map being contained in the image of $B$.  In this language, a section $S\in H^0(F, \Oc(1,1))$ will be defined by an equivariant map:

\begin{equation}\label{eq:hyperplanesection}
    \begin{tikzcd}[row sep=tiny, column sep = large, /tikz/column 1/.append style={anchor=base east} ,/tikz/column 3/.append style={anchor=base west}]
        \displaystyle (A, B) \ar{r}{S} & S_{i_1\dots i_k j_1\dots j_{k+1}} A_{[1}^{\,\,\,\,i_1}\cdots A_{k]}^{\,\,\,\,i_k}
        B_{[1}^{\,\,\,\,j_1}\cdots B_{k+1]}^{\,\,\,\,j_{k+1}}
    \end{tikzcd}
\end{equation}

where we used the ``physicist's'' convention of summing up repeated higher and lower indices, and enclosing by square brackets skew-symmetrized indices (e.g. $B_{[1}^{\,\,\,\,j} B_{2]}^{\,\,\,\,k} = B_{1}^{\,\,\,\,j} B_{2}^{\,\,\,\,k}-B_{2}^{\,\,\,\,j} B_{1}^{\,\,\,\,k}$ up to an irrelevant constant).

In the following, consider the functions $P: \Hom(V_k, V_n)\arw V_n^\vee$ and $Q:\Hom(V_{k+1}, V_n)\arw V_n$, where $P(A) = (P_1(A),\dots,P_n(A))$ and $Q(A) = (Q^1(A),\dots,Q^n(A))^T$ are given by: 
\begin{equation}
    \begin{split}
        P_r(A) &= S_{i_1\dots i_k j_1\dots j_{k+1}}A_{[1}^{\,\,\,\,i_1}\cdots A_{k]}^{\,\,\,\,i_k} \delta^{j_1}_{[r} A_{1}^{\,\,\,\,j_2}\cdots A_{k]}^{\,\,\,\,j_{k+1}}\\
        Q^r(B) &= S_{i_1\dots i_k j_1\dots j_{k+1}}B_{[1}^{\,\,\,\,i_1}\cdots B_{k}^{\,\,\,\,i_{k}}B_{k+1]}^{\,\,\,\,r} B_{[1}^{\,\,\,\,j_1}\cdots B_{k+1]}^{\,\,\,\,j_{k+1}}.
    \end{split}
\end{equation}

\begin{lemma}\label{lem:P_and_Q}
    The pushforwards of $S$ with respect to the maps $h_1:F(k,k+1, V_n)\arw G(k, V_n)$ and \linebreak $h_2:F(k,k+1, V_n)\arw G(k+1, V_n)$ are defined by the following functions:
    
    \begin{equation}
        \begin{tikzcd}[row sep=tiny, column sep = large, /tikz/column 1/.append style={anchor=base east} ,/tikz/column 3/.append style={anchor=base west}]
            \displaystyle A\ar{r}{h_{1*}S} & P(A), \\
            \displaystyle B\ar{r}{h_{*2}S} & Q(B), \\
        \end{tikzcd}
    \end{equation}
    
    which are the images of $h_{1*}S$ and $h_{2*}S$ respectively inside $H^0(G(k, V_n), V_n^\vee\otimes\Oc(2))$ and $H^0(G(k+1,V_n), V_n\otimes\Oc(2))$.
\end{lemma}

\begin{proof}
    Note that moving $A$ in $[A]$ means sending $A$ to $Ag^{-1}$ for some $g\in GL(V_k)$, and this rescales $P(A)$ by $\det g^{-2}$: hence, $P(A)$ is clearly an element of the fiber $V_n^\vee\otimes\Oc(2)_{[v]}$. Observe now that the dual of the tautological embedding, $V_n^\vee\otimes\Oc(2)\twoheadrightarrow \Uc^\vee(2)$, acts on the fiber over $[A]$ as the projection onto $\operatorname{im}(A)$. Choose now $v\in \operatorname{im}(A)$. Since the wedge product of $k+1$ vectors spanning a $k$-dimensional space vanishes, one has:
    
    \begin{equation}
            P_r(A) v^r = S_{i_1\dots i_k j_1\dots j_{k+1}}A_{[1}^{\,\,\,\,i_1}\cdots A_{k]}^{\,\,\,\,i_k} \delta^{j_1}_{[r} A_{1}^{\,\,\,\,j_2}\cdots A_{k]}^{\,\,\,\,j_{k+1}}v^r = 0
    \end{equation}
     
    and therefore $[A, P(A)]$ lies in the image of $\Qc^\vee(2)_{[A]}$ in $V^\vee\otimes\Oc(2)_{[A]}$. The case of $Q$ can be addressed in a similar way: we see that $[B, Q(B)]\in V_n\otimes\Oc(2)_{[B]}$ and we consider the projection $V_n\otimes\Oc(2)\twoheadrightarrow \Qc(2)$ which acts on the $[B]$-fiber by projecting on the vector space $\operatorname{im}(B)^\perp\subset V$. For any $v\in \operatorname{im}(B)^\perp$ we observe that:
     
     \begin{equation}
        \begin{split}
            Q^r(A) v_r = S_{i_1\dots i_k j_1\dots j_{k+1}}B_{[1}^{\,\,\,\,i_1}\cdots B_{k}^{\,\,\,\,i_{k}}B_{k+1]}^{\,\,\,\,r} B_{[1}^{\,\,\,\,j_1}\cdots B_{k+1]}^{\,\,\,\,j_{k+1}}v_r & = 0
        \end{split}
    \end{equation}
    because we contract a column of $B$ with $v$, and they are orthogonal. This proves that $[B, Q(B)]$ lies in the image of $\Uc(2)_{[B]}$ in $V_n\otimes\Oc(2)_{[B]}$. By computing the zeroes in coordinates, the zero locus of $S$ is exactly the locus where the projections along $h_1$ and $h_2$ increase their dimension over the zeroes of $h_{i*}(S)$.
\end{proof}

\begin{lemma}\label{lem:glsm_crit_loci}
    Consider the GLSM given by the data $(W, GL(V_{k+1}), \CC^*_R, f)$ where the $R$-symmetry acts trivially on $\Hom(V_{k+1}, V_n)$ and with weight two on $\Hom(V_{k+1}, \Sym^2\wedge^{k+1}V_{k+1}^\vee)$, and the superpotential $f$ is defined as:
    
    \begin{equation}\label{eq:superpotential}
        \begin{tikzcd}[row sep=tiny, column sep = large, /tikz/column 1/.append style={anchor=base east} ,/tikz/column 3/.append style={anchor=base west}]
            W \ar{r}{f} & \CC \\
            (B, \omega) \ar[maps to]{r} & h_{2*}S(B)\cdot\omega
        \end{tikzcd}
    \end{equation}
    
    where $S\in H^0(F(k, k+1, V_n), \Oc(1,1))$ is chosen such that $h_{2*}S$ is the section of $\Uc(2)$ whose image in $V\otimes\Oc(2)$ is $Q$. Then one has the following isomorphisms:
    \begin{equation}
        \operatorname{Crit}(f)\git_+GL(V_{k+1}) \simeq Y_+ = Z(h_{2*}S), \,\,\,\,\, \operatorname{Crit}(f)\git_-GL(V_{k+1}) \simeq Y_- = Z(h_{1*}S).
    \end{equation}
    
\end{lemma}

\begin{proof}
    The statement about $Y_+$ is simply a consequence of \cite[Remark 1.2]{okonekteleman}. To prove the second statement, we just need to show that the superpotential restricted to $\wt W$ (see Equation \ref{eq:git_quotient-_stab} and the discussion above) can be written as:
    
    \begin{equation}\label{eq:superpotential_minus_side}
        \begin{tikzcd}[row sep=tiny, column sep = large, /tikz/column 1/.append style={anchor=base east} ,/tikz/column 3/.append style={anchor=base west}]
            \wt W = \{(v, C)\in V_n\oplus\Hom(V_k, V_n) : \rk C =k\} \ar{r}{f} & \CC \\
            (v, C) \ar[maps to]{r} & v^r h_{1*}S(C)_r.
        \end{tikzcd}
    \end{equation}
    
    In fact, we showed that $X_- = \wt W/\operatorname{Stab}\simeq\operatorname{tot}(\wt\Qc(2)\arw G(k, V_{n}))$, and this together with Equation \ref{eq:superpotential_minus_side} and \cite[Remark 1.2]{okonekteleman} would prove that $\crit(f)\git_- GL(V_k)\simeq Z(h_{1*}S)= Y_-$. \\
    \\
    Let us begin by writing the superpotential explicitly. By  Lemma \ref{lem:P_and_Q} we have the following identity:
    
    $$
    B_\alpha^{\,\,\,\, r} h_{2*}S^\alpha(B) = Q^r(B) = S_{i_1\dots i_k j_1\dots j_{k+1}}B_{[1}^{\,\,\,\,i_1}\cdots B_{k}^{\,\,\,\,i_{k}}B_{k+1]}^{\,\,\,\,r} B_{[1}^{\,\,\,\,j_1}\cdots B_{k+1]}^{\,\,\,\,j_{k+1}}
    $$
    
    from which we deduce that, up to an overall sign independent of $r$:
    
    $$
    h_{2*}S^\alpha(B)  = S_{i_1\dots i_k j_1\dots j_{k+1}}\delta^\alpha_{[1}B_{2}^{\,\,\,\,i_1}\cdots B_{k+1]}^{\,\,\,\,i_{k}} B_{[1}^{\,\,\,\,j_1}\cdots B_{k+1]}^{\,\,\,\,j_{k+1}}.
    $$
    
    Plugging this into Equation \ref{eq:superpotential} one obtains:
    
    \begin{equation}
        f(B, \omega) = \omega_\alpha S_{i_1\dots i_k j_1\dots j_{k+1}}\delta^\alpha_{[1}B_{2}^{\,\,\,\,i_1}\cdots B_{k+1]}^{\,\,\,\,i_{k}} B_{[1}^{\,\,\,\,j_1}\cdots B_{k+1]}^{\,\,\,\,j_{k+1}}.
    \end{equation}
    
    Let us now evaluate $f$ on $\wt W$, i.e. we fix $\omega = (1, 0, \dots, 0)$ and $B = (v|C)$ with $\rk(C) = k$. What we find is:
     
    \begin{equation}
        f(v, C):= f((v|C), (1, 0,\dots,0)) = S_{i_1\dots i_k j_1\dots j_{k+1}}C_{[1}^{\,\,\,\,i_1}\cdots C_{k]}^{\,\,\,\,i_{k}} \delta^{j_1}_{[r}v^{r}C_{1}^{\,\,\,\,j_2}\cdots C_{k]}^{\,\,\,\,j_{k+1}} = v^r P_r(C).
    \end{equation}
    
    Since $P$ is the image in $H^0(G(k, V_n), V_n\otimes\Oc(2))$ of $h_{1*}S$, there exists a basis $e_1,\dots,e_{n-k}$ of $\im(C)^\perp$ such that $E_r^{\,\,t} h_{2*}(C)_t = P_r(C)$, where on $E:=(e_1|\dots|e_{n-k})$ the $\operatorname{Stab}$-action is trivial. Hence:
    
    \begin{equation}
        f(v, C) =  v^r E_r^{\,\,t} h_{2*}(C)_t,
    \end{equation}
    
    and since the map $E:v\arw vE$ is the projection from the fiber of $V_n\otimes\Oc(-2)$ onto the fiber of $\wt\Qc(-2)$ the proof is concluded.
\end{proof}

\subsection{Window categories}\label{subsection:windows}

To fix the notation we will briefly review the representation algebra and the Borel--Weil--Bott theorem for Grassmannians of type $A$. For a more detailed exposition on this formulation of the Borel--Weil--Bott theorem, see \cite[Appendix A]{bpc}.

Fix a vector space $V_{ {n}}\simeq \CC^n$. Given a rank $r$ vector bundle $\Ec$ and a sequence $\lambda = (\lambda_1,\dots, \lambda_r)$ of non-increasing integers, we denote by $\Sc_\lambda\Ec$ the image of $\Ec$ through the Schur functor of $GL(r)$ associated to $\lambda$. For example, one has $\wedge^k\Ec = \Sc_{(1^k, 0^{r-k})}\Ec$ and $\Sym^k\Ec = \Sc_{(k, 0^{r-1})}\Ec$.

Every completely reducible homogeneous vector bundle on $G(k, V)$ can be written as sum of terms of the form $\Sc_\lambda\Uc^\vee\otimes \Sc_\delta\Qc^\vee$, where $\Uc$ and $\Qc$ are the tautological and quotient vector bundles of the sequence \ref{eq:tautologicalseq}. We associate to $\Sc_\lambda\Uc^\vee\otimes \Sc_\delta\Qc^\vee$ a double sequence of integers $(\lambda|\delta):=(\lambda_1, \dots, \lambda_k  {|} \delta_1\dots, \delta_{n-k})$. In this language, the Borel--Weil--Bott theorem can be rewritten as:

\begin{theorem}[Borel--Weil--Bott for Grassmannians of type $A$]\label{thm:borelweilbott}
    Let $V_{ {n}}$ be a vector space of dimension $n$ and let $\lambda$ and $\delta$ be non-increasing sequences of integers of lengths respectively $k$ and $n-k$. Fix $\rho = (n-1, n-2,\dots,1,0)$. Then one and only one of the following possibilities occur:
    
    \begin{enumerate}
    
        \item[$\circ$] The sequence $(\lambda|\delta)+\rho$ contains a repeated number. Then $\Sc_\lambda\Uc^\vee\otimes \Sc_\delta\Qc^\vee$  {has no cohomology}.
         
        \item[$\circ$] The sequence $(\lambda|\delta)+\rho$ contains no repeated numbers. Then there exists a permutation $P$ of minimal length $l$ such that $P((\lambda|\delta)+\rho)$ is decreasing, and one has $H^\bullet(G(k, V_{ {n}}), \Sc_\lambda\Uc^\vee\otimes \Sc_\delta\Qc^\vee) = \Sc_{P((\lambda|\delta)+\rho)-\rho}V_{ {n}}[-l]$ where $\Sc_{P((\lambda|\delta)+\rho)-\rho}V_{ {n}}$ is the representation of $GL(V_{ {n}})$ of highest weight $P((\lambda|\delta)+\rho)-\rho$.
    \end{enumerate}
\end{theorem}

To each double weight $(\lambda|\delta)$ we can associate a double Young tableau. For instance, on $G(4, 6)$ one has:

\begin{equation}
    (3,2,2,1 | 2,0) \rightarrow {\tiny \frac{\ydiagram{3,2,2,1}}{\ydiagram{2,0}\hspace{9pt}}}.
\end{equation}

By Theorem \ref{thm:borelweilbott} and the Littlewood--Richardson formula, we are able to compute cohomology for all completely reducible homogeneous vector bundles on Grassmannians. Note that the Littlewood--Richardson formula must be applied to the components above and below the bar \emph{separately}, as illustrated by the following example on $G(2,  {V_n})$ for $n>1$:

\begin{equation}
    \begin{split}
    \tiny
        \frac{\ydiagram{2,0}}{\hspace{-10pt}\ydiagram{1,0}}\otimes \frac{\ydiagram{1,1}}{\ydiagram{1,0}} = \frac{\ydiagram{3,1}}{\hspace{-10pt}\ydiagram{2,0}}\oplus \frac{\ydiagram{3,1}}{\hspace{-18pt}\ydiagram{1,1}}.
    \end{split}
\end{equation}

We introduce the following notation:

\begin{equation}\label{eqn:box}
    \operatorname{Box}(n-k, k) = \left\{ \lambda\in\ZZ^{n-k} ~ : ~ k\geq\lambda_1\geq\cdots\geq\lambda_{n-k} \right\}.
\end{equation}

Let us recall the following definition:

\begin{definition}(cf. \cite[Section 7]{hillevandenbergh}) 
    Let $X$ be a quasi-projective variety and $T$ a finite rank vector bundle. We say $T$ is \emph{partially tilting} if $\Ext_{X}^{>0}(T,T) = 0$. Moreover, we say $T$ is \emph{tilting} if it is partially tilting and it generates the unbounded category of quasi-coherent sheaves $D(\operatorname{Qcoh}X)$ (i.e. for any object $E\in D(\operatorname{Qcoh}X)$, $R\Hom(E, T) = 0$ implies $E = 0$).
\end{definition}

In the following, we need to construct a vector bundle $\Tc$ on $\Xc$ such that its pullbacks to $X_\pm$ are tilting bundles. Recall that $X_+$ is the total space of a vector bundle over $G(n-1, V_n)$, hence it would be natural to choose $\Tc$ such that $i_+^*\Tc$ is a direct sum of pullbacks of homogeneous vector bundles from $G(n-1, V_n)$. It is well-known that {the direct sum $T_K$ of the bundles appearing in Kapranov's full exceptional collection \cite{kapranov}} is tilting  on any Grassmannian. However, such property is in general not preserved by pullback to the total space of a vector bundle, and there is in principle no reason why the associated bundle $\Tc$ on  $\Xc$ would restrict to a tilting bundle on $X_-$, which is in fact what happens if we choose $\Tc$ such that $i_+^*\Tc \simeq T_K$. Therefore we construct a different candidate:

\begin{definition}
    \label{def:mutated_bundle}
    We call \emph{mutated Kapranov's bundle} the vector bundle $T_{M}$ on $G(n-1, V)$ given by:
    
    \begin{equation}
        T_{M}:= \bigoplus_{\lambda\in\operatorname{Box}(n-3, 1)}\Sc_\lambda  {\Uc^\vee}\oplus\Oc(1)\oplus\Oc(2)
    \end{equation}
    
\end{definition}

Note that a homogeneous vector bundle on ${G(n-1, V_n)}$, corresponding to a representation of weight $\omega$, extends to a vector bundle on the space $\Xc$. We will call $E_\omega$ such object on $\Xc$. This motivates the following:

\begin{definition}\label{def:mutated_window}
    Define the set $I_W:=\left\{\operatorname{Box}(n-3, 1)
        \cup\{(1, \dots, 1); (2, \dots, 2)\}\right\}$.
    We call \emph{mutated window} the subcategory $\Wc_0\subset \dbcoh(\Xc)$ generated by the objects $\left\{E_{(\lambda|0)} : \lambda\in I_W\right\}$.\\
    \\
    Moreover, we define:
    
    \begin{equation}
        \Tc:= \bigoplus_{\lambda\in I_W} E_{(\lambda|0)}
    \end{equation}
    
\end{definition}

\begin{remark}
    The reason of these names in Definitions \ref{def:mutated_bundle} and \ref{def:mutated_window} is to recall the fact that while Kapranov's bundle is the direct sum of the objects appearing in Kapranov's full exceptional collection:
    
    \begin{equation}
        \dbcoh(G(n-1, V_n)) = \langle\Oc, \Uc^\vee,\dots, \wedge^{n-2}\Uc^\vee, \Oc(1)\rangle,
    \end{equation}
    
    the mutated Kapranov bundle is the direct sum of the objects appearing in the following collection:
    
    \begin{equation}\label{eq:mutated_collection}
        \dbcoh(G(n-1, V_n)) = \langle\Oc, \Uc^\vee,\dots, \wedge^{n-3}\Uc^\vee, \Oc(1), \Oc(2)\rangle.
    \end{equation}
    
    The latter can be obtained by the former, mutating $\wedge^{n-2}\Uc^\vee$ one step to the right: in fact, $\wedge^{n-2}\Uc^\vee\simeq\Uc(1)$, and by the tautological sequence $\RR_{\Oc(1)}\Uc(1) \simeq \Qc(1) \simeq \Oc(2)$ up to an irrelevant shift. This also tells that Equation \ref{eq:mutated_collection} is indeed a full exceptional collection.
\end{remark}

\begin{proposition}\label{prop:equivalence+_no_potential}
    The functor $i_{ {+}}|_{\Wc_0}^*:\Wc_0\arw\dbcoh(X_{ {+}})$ is an equivalence of categories.
\end{proposition}

\begin{proof}
    We will need a vanishing result (Lemma \ref{lem:partiallytilting+}), which, in order to keep the proof readable, will be discussed after it. Let us first prove that $i_+|_{\Wc_0}^*$ is fully faithful: this amounts to show that for every $\lambda, \lambda'\in I_W$:
    
    \begin{equation}\label{eq:ext_equality+}
        \Ext^\bullet_\Xc (E_{(\lambda|0)}, E_{(\lambda'|0)}) = \Ext^\bullet_{X_+} (i_+|_{\Wc_0}^*E_{(\lambda|0)}, i_+|_{\Wc_0}^*E_{(\lambda'|0)}).
    \end{equation}
    
    First, observe that $i_+$ is an isomorphism outside a locus of codimension at least two: hence, by Hartogs's lemma, it preserves global sections of sheaves, and therefore  Equation \ref{eq:ext_equality+} holds in degree zero. This reduces the proof of fully faithfulness to showing that both sides of Equation \ref{eq:ext_equality+} are zero in higher degrees, i.e. that there are no higher extensions between direct summands of both $\Tc$ and $\pi_+^*T$, which can be rephrased saying that both bundles are partially tilting. Since $\Xc$ is a quotient algebraic stack by a linear reductive group, \cite[Lemma 2.2.8]{ballardfaverokatzarkov} holds. In particular there are no higher Exts between any two vector bundles in $\dbcoh\Xc$, which implies that $\Tc$, being a direct sum of vector bundles, is partially tilting. On the other hand, by Lemma \ref{lem:partiallytilting+} the bundles $i_+|_{\Wc_0}^* \Tc=\pi_+^*T$ is partially tilting as well.\\
    \\
    Let us now prove that $i_+|_{\Wc_0}^*$ is essentially surjective. In the following, given a sheaf $F$ on $X_+$, we will call a \emph{lift} of $F$ an object $\Fc$ on $\Xc$ such that $\Fc$ restricts to $F$.\\
    Since $i^*_+|_{\Wc_0}:\Wc_0\arw\dbcoh(X_+)$ is fully faithful, one has an equivalence of categories $\Wc_0\arw\Ic m$, where $\Ic m$ is the essential image of $i^*_+|_{\Wc_0}$. Moreover, since $i^*_+|_{\Wc_0}$ is exact, and generators of $D^b(X_+)$ are contained in its essential image it follows that $i^*_+|_{\Wc_0}$ is essentially surjective.

    More precisely, by definition of $\Wc_0$, all direct summands of $\pi_+^*T$ are objects of $\Ic m$. On the other hand, these summands generate the whole category $\dbcoh(X_+)$. Hence, to conclude the proof we just need to show that also cones, shifts and direct sums of objects that can be lifted to $\Wc_0$ can also be lifted to objects of $\Wc_0$. Taking shifts and direct sums commutes with $i^*_+|_{\Wc_0}$ and thus shifts and direct sums lift to shift and direct sums of their lifts. Consider now a morphism $f:i^*_+A\arw i^*_+B$, with $A,B$ objects from $\Wc_0$. Since $i^*_+|_{\Wc_0}$ is fully faithful, there is a map $\wt f: A\arw B$
    such that $i^*_+|_{\Wc_0} \wt f=f$, and we simply define the lift of $\operatorname{Cone}(f)$ to be $\operatorname{Cone}(\wt f)$. This is indeed a lift since it comes with a distinguished triangle
    
    \begin{equation}
        A\arw B\arw\operatorname{Cone}(\wt f)\arw A[1]
    \end{equation}
    
    which proves that $\operatorname{Cone}(\wt f)\in\Wc_0$ and by exactness of $i^*_+$ we have $i^*_+|_{\Wc_0}(\operatorname{Cone}(\wt f))= \operatorname{Cone}(f)$.

\end{proof} 

\begin{lemma}\label{lem:partiallytilting+}
    The bundle $\pi_{ {+}}^*T$ is partially tilting, where $\pi_{ {+}}:X_{ {+}}\arw  {G(V_{n-1}, V_n)}$ is the vector bundle projection.
\end{lemma}

\begin{proof}
    One has:

    \begin{equation}\label{eq:extsleftside1}
        \begin{split}
            \Ext^{\bullet}_{X_{ {+}}}(\pi_{ {+}}^*\Sc_\lambda {\Uc^\vee}, \pi_{ {+}}^*\Sc_{\lambda'} {\Uc^\vee}) & = \Ext^{\bullet}_{ {G(V_{n-1}, V_n)}}(\Sc_\lambda {\Uc^\vee}, \pi_{{ {+}}*}(\pi_{ {+}}^*\Sc_{\lambda'} {\Uc^\vee}\otimes\Oc_{X_{ {+}}}))\\
            & = \Ext^{\bullet}_{ {G(V_{n-1}, V_n)}}(\Sc_\lambda {\Uc^\vee}, \Sc_{\lambda'} {\Uc^\vee}\otimes\pi_{{ {+}}*}\Oc_{X_{ {+}}})\\
            & = \Ext^{\bullet}_{ {G(V_{n-1}, V_n)}}(\Sc_\lambda {\Uc^\vee}, \Sc_{\lambda'} {\Uc^\vee}\otimes\bigoplus_{m\geq 0}\operatorname{Sym}^{ {m}}( {\Uc}(2)))\\
            & = \bigoplus_{m\geq 0}H^\bullet( {G(V_{n-1}, V_n)}, (\Sc_\lambda {\Uc^\vee})^\vee\otimes\Sc_{\lambda'} {\Uc^\vee}\otimes\operatorname{Sym}^{ {m}}( {\Uc}(2))).
        \end{split}
    \end{equation}
    
    Observe now that if $\lambda = (\lambda_1, \dots\lambda_{n-1})$ one has 
    
    \begin{equation}
        (\Sc_\lambda {\Uc^\vee})^\vee = \Sc_\lambda {\Uc^\vee} = \Sc_{\bar\lambda} {\Uc^\vee}( {-}\lambda_1) 
    \end{equation}
    
    where we defined $\bar\lambda :=(-\lambda_{n-1}+\lambda_1, -\lambda_{n-2}+\lambda_1, \dots, -\lambda_2+\lambda_1, 0)$. This allows us to rewrite Equation \ref{eq:extsleftside1} as:
    
    \begin{equation}\label{eq:extsleftside2}
        \begin{split}
            \Ext^{\bullet}_{X_{ {+}}}(\pi_{ {+}}^*\Sc_\lambda {\Uc^\vee}, \pi_{ {+}}^*\Sc_{\lambda'} {\Uc^\vee}) & =
            \bigoplus_{m\geq 0}H^\bullet( {G(V_{n-1}, V_n)}, \Sc_{\bar\lambda} {\Uc^\vee}\otimes\Sc_{\lambda'} {\Uc^\vee}\otimes\operatorname{Sym}^m {\Uc}(2m {-}\lambda_1))
        \end{split}
    \end{equation}

    We will first consider the situation when $\lambda' = (2,\dots,2)$ which corresponds to computing:
    
    \begin{equation}
        \begin{split}
            \Ext^{\bullet}_{X_{+}}(\pi_{+}^*\Sc_\lambda\Uc^\vee, \pi_{+}^*\Oc(2)) & =
            \bigoplus_{m\geq 0}H^\bullet(G(V_{n-1}, V_n), \Sc_{\bar\lambda}\Uc^\vee\otimes\operatorname{Sym}^m\Uc(2m-\lambda_1+2))
        \end{split}
    \end{equation}
    
    The condition $\lambda_i\leq 2$ for every $i$ implies $\bar\lambda_i\leq 2$, therefore the argument of the right hand side is a direct sum of products of irreducible, globally generated homogeneous vector bundles on $G(n-1, V_n)$, which cannot have higher cohomology.\\
    
    Let us now check the case $\lambda_i, \lambda_i'\leq 1$ for every $i$. The condition $\lambda_i\leq 1$ for every $i$ implies $\bar\lambda_i\leq 1$. Note also that  {$\operatorname{Sym}^m\Uc(k) = \Sc_{(m,m,\dots,m,0)}\Uc^\vee\otimes\Oc(k-m)$}. We can recast Equation \ref{eq:extsleftside2} once again, finding: 
    
    \begin{equation}\label{eq:extsleftside3}
        \begin{split}
            \Ext^{\bullet}_{X_{ {+}}}(\pi_{ {+}}^*\Sc_\lambda {\Uc^\vee}, \pi_{ {+}}^*\Sc_{\lambda'} {\Uc^\vee}) & =
            \bigoplus_{m\geq 0}H^\bullet( {G(V_{n-1}, V_n)}, \Sc_{\delta} {\Uc^\vee}( {m -}\lambda_1))
        \end{split}
    \end{equation}
    
    where $\delta$ is a weight given by the product of two weights $\bar\lambda, \lambda'$  {with all entries lesser or equal than one}, together with a third weight  {$(m, m, \dots, m, 0)$} and the proof reduces to showing that each summand in the right hand side of \ref{eq:extsleftside3} has no higher cohomology.

    Let us start by proving the statement for $m=0$. The product $\Sc_{\bar\lambda}\Uc^\vee\otimes \Sc_{\lambda'}\Uc^\vee(\lambda_1)$ of Equation \ref{eq:extsleftside2} can be expanded by applying the Littlewoods-Richardson formula to diagrams of the following shape:
    
    \begin{equation}\label{eq:young_diagram_vanishing_lemma}
        \begin{split}
            \tiny
            \frac{\ydiagram{0}}{\underbrace{\ydiagram{1}}_{\lambda_1}}\otimes
            \frac{\ydiagram{1,1,1,1,0}}{\ydiagram{0}}\otimes
            \frac{\ydiagram{1,1,1,0,0}}{\ydiagram{0}}
        \end{split}
    \end{equation}
    
    The result is a sum of diagrams with weights of the form $(\gamma_1,\dots\gamma_{n-1}|\lambda_1)$ where $\gamma_i\leq 1+\lambda_1$. Two possibilities can occur:
    \begin{enumerate}
        \item[$\circ$] $\lambda_1 = 0$. Let us apply Theorem \ref{thm:borelweilbott}: since $\gamma_1\geq\dots\geq\gamma_{n-1}\geq\lambda_1 = 0$ the cohomology of the associated bundle is concentrated in degree zero.

        \item[$\circ$] $\lambda_1 = 1$. If $\gamma_{n-1}\geq 1$ the cohomology is concentrated in degree zero as above, while if $\gamma_{n-1} = 0$, in the notation of Theorem \ref{thm:borelweilbott}, once we add $\rho$ to the weight there is a repetition between the last two entries and hence the cohomology is zero in all degrees.
    \end{enumerate}
    
    The case $m>0$ can be treated by multiplying to the decomposition of Equation \ref{eq:young_diagram_vanishing_lemma} the following diagram:
     
    \begin{equation}\label{eq:young_diagram_symmetric_power}
        {\tiny
        \frac{\overbrace{\ydiagram{8,8,8,8,8,4}}^{2m}}
        {\ydiagram{0}}
        }
    \end{equation}
    In every summand of the decomposition of Equation \ref{eq:young_diagram_vanishing_lemma}, the length of the row above the bar is either equal to the one below or shorter by one box. Therefore, by applying the Littlewood--Richardson formula, multiplying with the diagram \ref{eq:young_diagram_symmetric_power} for any positive $m$ gives a non-decreasing weight: hence, by Theorem \ref{thm:borelweilbott}, no cohomology can occur in degree higher than zero.\\
    \\
    The remaining case to consider is the one where $\lambda = (2, \dots, 2)$ and $\lambda' = (1, \dots, 1, 0, \dots, 0)$, i.e.
    
    \begin{equation}
        \begin{split}
            \Ext^{\bullet}_{X_{+}}(\pi_{+}^*\Oc(2), \pi_{+}^*\Sc_{\lambda'}\Uc^\vee) & =
            \bigoplus_{m\geq 0}H^\bullet(G(V_{n-1}, V_n), \Sc_{\lambda'}\Uc^\vee\otimes\operatorname{Sym}^m\Uc(2m-2))
        \end{split}
    \end{equation}
    
    Each summand has an associated diagram of one of the following shapes (here illustrated for $n=7$):
    
    \begin{equation}
        \begin{split}
            \tiny
            \frac{\hspace{-10pt}\ydiagram{1,1,1,1,0,0}}{\ydiagram{2}}\otimes\frac{\overbrace{\ydiagram{8,8,8,8,8,4}}^{2m}}
            {\ydiagram{0}}\hspace{3pt}; \hspace{25
            pt}
            \tiny
            \frac{\hspace{-10pt}\ydiagram{1,1,1,1,1,1}}{\ydiagram{2}}\otimes\frac{\overbrace{\ydiagram{8,8,8,8,8,4}}^{2m}}
            {\ydiagram{0}}
        \end{split}
    \end{equation}
    
    As above, for every value of $m\geq 0$ we obtain either sections or a repetition.
    
\end{proof}

\subsubsection{Vector bundles and variation of GIT}

The next step of our construction of a derived embedding $\dbcoh(X_+)\subset\dbcoh(X_-)$ is to prove the following statement:

\begin{proposition}
    The functor $i_-|_{\Wc_0}^*:\Wc_0\arw\dbcoh(X_-)$ is fully faithful.
\end{proposition}

As in the proof of Proposition \ref{prop:equivalence+_no_potential}, we just need to show that $i_-^*\Tc$ is partially tilting:

\begin{proposition}\label{prop:partiallytilting-}

    The vector bundle $i_-^*\Tc$ is partially tilting, i.e one has:
     
    \begin{equation}
        \Ext^{>0}_{X_-}(i_-^*\Tc, i_-^*\Tc) = 0
    \end{equation}
    
\end{proposition}

\begin{proof}
    This is an immediate consequence of the discussion below and the vanishings of Lemma \ref{lem:difficultvanishing1} and Lemma \ref{lem:difficultvanishing2} following it.
\end{proof}

First, the bundle $E_{(\lambda|0)}$ can be described as follows:

\begin{equation}\label{eq:bundles_on_stack}
    E_{(\lambda|0)} = (W\oplus V_\lambda)/GL(V_{n-1})
\end{equation}

where $W$ is the vector space defined in Lemma \ref{lem:git_quotients} and $V_\lambda$ is a vector space on which $GL(V_{n-1})$ acts via the representation of weight $\lambda$. Clearly, by construction, $i_+^*E_{(\lambda|0)}$ is the pullback from $G(n-1, V_n)$ of $\Sc_\lambda\Uc^\vee$, i.e. the following:

\begin{equation}
    \pi_+^*\Sc_\lambda\Uc^\vee = \left(W\setminus Z_+\times V_\lambda\right) / GL(V_{n-1})
\end{equation}

By Equation \ref{eq:bundles_on_stack} we see that $i_-^*E_{(\lambda|0)}$ is the bundle described as the quotient by the same $GL(V_{n-1})$-action, but restricted to $W\setminus Z_-\times V_\lambda$:

\begin{equation}\label{eq:bundle_general_vgit}
    i_-^*E_{(\lambda|0)} = \left(W\setminus Z_-\times V_\lambda\right) / GL(V_{n-1})
\end{equation}

By our definition of $\Wc_0$, we are mainly interested in describing $i_-^*E_{(\lambda|0)}$ as explicitly as possible for $\lambda$ such that either $i_+^*E_{(\lambda|0)} = \pi_+^*\wedge^p\Uc^\vee$ or $i_+^*E_{(\lambda|0)} = \pi_+^*\Oc(m)$. Let us start with the former: by the discussion of Section \ref{subsection:glsm}, we can reduce the $GL(V_{n-1})$-action on $W\setminus Z_-$ to the action of a subgroup $\operatorname{Stab}\subset GL(V_{n-1})$ on $\wt W := \{(v, C)\in V_n\oplus\Hom(V_{n-2}, V_n) : \rk C = n-2\}$. If we apply the same procedure to $i_-^*E_{(\lambda|0)}$ we find:

\begin{equation}\label{eq:vgit_tautological_dual}
    i_-^*E_{(\lambda|0)} = \left(\wt W \times\Hom(\wedge^{p}V_{n-1}, \CC) \right) / \operatorname{Stab}
\end{equation}

where the action of $g\in\operatorname{Stab}$ on the second factor is by precomposition (i.e. multiplication from the right) with the matrix $\wedge^{p}g^{-1}$ of minors of $g$ with order $p$. Let us fix, for now, $p=1$, i.e. $
\lambda = (1,0\dots,0)$. Given the decomposition of $g^{-1}\in\operatorname{Stab}$ as in Equation \ref{eq:stab}, the corresponding $\operatorname{Stab}$-action on $(v, C, w)\in\wt W\times \Hom(\wedge^{p}V_{n-1}, \CC)$ is:

\begin{equation}
    g, (v, C, w) \longmapsto (v\det h^2 + C \kappa, Ch^{-1}, w g^{-1}).
\end{equation}

Let us now consider the following bundles, where $GL(V_{n-2})\subset\operatorname{Stab}$ acts on $V_{n-2}$ via the fundamental representation:

\begin{equation}
    \pi_-^*\wt\Uc^\vee = \left(\wt W \times\Hom(V_{n-2}, \CC) \right) / \operatorname{Stab}
\end{equation}

\begin{equation}
    \pi_-^*\Oc^\vee(-2) = \left(\wt W \times\Sym^2(\wedge^{n-2}V_{n-2}^\vee) \right) / \operatorname{Stab}
\end{equation}

where we recall that $\wt\Uc$ is the tautological bundle of $G(n-2, V_n)$. We will now show that $i_-^*E_{(1,0,\dots,0|0)}$ is an extension of these bundles above.\\
Given $v = (v_1, \dots, v_{n-2}, 0):=(w, 0)\in V_{n-1}^\vee$, the action by right-multiplication of $g^{-1}$ will be $(w, 0)\longmapsto(w h^{-1} \det h^2, 0)$. If we see $w$ as an element of $V_{n-2}\otimes\Sym^2\wedge^{n-2}V_{n-2}$, this gives an equivariant embedding of vector spaces $V_{n-2}\otimes\Sym^2\wedge^{n-2}V_{n-2}\xhookrightarrow{\,\,\,\,\,\,\,}V_{n-1}^\vee$ which defines an embedding of vector bundles $\pi_-^*\wt\Uc^\vee(-2)\xhookrightarrow{\,\,\,\,\,\,\,}i_-^*E_{(1,0,\dots,0|0)}$ on $X_-$. By the same kind of argument we see that $i_-^*E_{(1,0,\dots,0|0)}$ surjects on the cokernel $\pi_-^*\wt\Uc^\vee$, yielding the following short exact sequence of vector bundles on $X_-$:

\begin{equation}\label{eq:vgitsequencep=1}
    0\arw\Oc(-2)\arw i_-^*E_{(1,0,\dots,0|0)}\arw\pi_-^*\wt\Uc^\vee\arw 0.
\end{equation}

For higher $p$, note that $E_{(1,\dots, 1,0,\dots,0|0)} = \wedge^p E_{(1,0,\dots,0|0)}$: therefore we just need to take the $p$-th wedge power of Equation \ref{eq:vgitsequencep=1}, obtaining:

\begin{equation}\label{eq:vgitsequence}
    0\arw\pi_-^*\wedge^{p-1}\wt\Uc^\vee(-2)\arw i_-^*E_{(1,\dots, 1,0,\dots,0|0)}\arw\pi_-^*\wedge^{p}\wt\Uc^\vee\arw 0.
\end{equation}

Consider now $\lambda = (m, 0, \dots, 0)$. Since $\pi_+^*\Oc(m) = \pi_+^*\Sym^m\wedge^{n-1}\Uc^\vee$, by recasting Equation \ref{eq:bundle_general_vgit} in terms of a quotient by $\operatorname{Stab}$, with the appropriate $\lambda$, we obtain:

\begin{equation}\label{eq:vgit_line_bundle}
    i_-^*E_{(m,0,\dots,0|0)} = \left(\wt W \times\Sym^m\wedge^{n-1}V_{n-1}^\vee \right) / \operatorname{Stab}.
\end{equation}

The action of $g\in\operatorname{Stab}$ in this case is: 

\begin{equation}
    g, (v, C, w) \longmapsto (v\det h^2 + C \kappa, Ch^{-1}, w\det g^{-m}),
\end{equation}

and by $\det g = \det h^{-1}$ this allows to conclude that $i_-^*E_{(m,0,\dots,0|0)} = \pi_-^*\Oc(-m)$.

Fix now $\lambda(p):=(1,\dots,1,0,\dots,0)$ where $1$ appears $p$ times. In the following Lemmas \ref{lem:difficultvanishing1} and \ref{lem:difficultvanishing2} we will need to prove some vanishing results about cohomology of vector bundles on $X_-$. We will extensively make use of the following sequence for $1\leq p\leq q\leq n-1$:

\begin{equation}\label{eq:sequencewithH1}
    \begin{split}
        0\arw \pi_-^*(\wedge^{n-p-2}\wt\Uc^\vee\otimes\wedge^{q-1}\wt\Uc^\vee(-3))
        \arw
        (i_-^*E_{(\lambda(p)|0)})^\vee\otimes\pi_-^*\wedge^{q-1}\wt\Uc^\vee(-2)
        \arw \hspace{50pt} \\
        \arw \pi_-^*(\wedge^{n-p-1}\wt\Uc^\vee\otimes\wedge^{q-1}\wt\Uc^\vee(-1)) \arw 0
    \end{split}
\end{equation}

which is the tensor product of the dual of the sequence \ref{eq:vgitsequence} with $\pi_-^*\wedge^{q-1}\wt\Uc^\vee(-2)$. Let us begin with a technical statement, which will be crucial in the proof of the vanishings.

\begin{lemma}\label{lem:lifesavinglemma}

    Consider the sequence \ref{eq:sequencewithH1} and suppose the following conditions hold:

    \begin{itemize}
        \item[$\circ$] $U_{(0)}:=H^0_{(0)}(X_-, \pi_-^*(\wedge^{n-p-1}\wt\Uc^\vee\otimes\wedge^{q-1}\wt\Uc^\vee(-1))) \simeq H^1(X_-, \pi_-^*(\wedge^{n-p-2}\wt\Uc^\vee\otimes\wedge^{q-1}\wt\Uc^\vee(-3)))$
        \item[$\circ$] $H^{>0}(X_-, \pi_-^*(\wedge^{n-p-1}\wt\Uc^\vee\otimes\wedge^{q-1}\wt\Uc^\vee(-1))) = 0$
        
        \item[$\circ$] $H^{>1}(X_-, \pi_-^*(\wedge^{n-p-2}\wt\Uc^\vee\otimes\wedge^{q-1}\wt\Uc^\vee(-3))) = 0$
    \end{itemize}

    where the left hand side of the first equation denotes the space of sections which are constant on the fibers of $\pi_-$. Then $H^1(X_-, (i_-^*E_{(\lambda(p)|0)})^\vee\otimes\pi_-^*\wedge^{q-1}\wt\Uc^\vee(-2)) = 0$.
    
\end{lemma}

\begin{proof}

    By the conditions above, the long exact sequence of cohomology associated to the sequence \ref{eq:sequencewithH1} reads:
    
    \begin{equation}
        \begin{split}
            0\arw H^0(\wedge^{n-p-2}\wt\Uc^\vee\otimes\wedge^{q-1}\wt\Uc^\vee(-3)) \arw 
            H^0((i_-^*E_{(\lambda(p)|0)})^\vee\otimes\wedge^{q-1}\wt\Uc^\vee(-2)) \xrightarrow{\,\,\beta_0\,\,} \hspace{50pt}\\
            H^0(\wedge^{n-p-1}\wt\Uc^\vee\otimes\wedge^{q-1}\wt\Uc^\vee(-1))\arw
            U_{(0)} \arw
            H^1((i_-^*E_{(\lambda(p)|0)})^\vee\otimes\wedge^{q-1}\wt\Uc^\vee(-2)) \arw 0
        \end{split}
    \end{equation}
    
    which tells us that $\operatorname{coker}\beta_0\subseteq U_{(0)}$. If we prove that the space of all global sections of $\pi_-^*(\wedge^{n-p-1}\wt\Uc^\vee\otimes\wedge^{q-1}\wt\Uc^\vee(-1))$ which are constant of the fibers of $\pi_-$ is isomorphic to a subspace of $\operatorname{coker}\beta_0$, we find $\operatorname{coker}\beta_0\simeq U_{(0)}$, and therefore $H^1((i_-^*E_{(\lambda(p)|0)})^\vee\otimes\wedge^{q-1}\wt\Uc^\vee(-2)) = 0$.

    The second morphism in the sequence \ref{eq:sequencewithH1} can be also rewritten as follows:
    
    \begin{equation}\label{eq:surjection_hom_bundles}
        \beta:\Hc om(i_-^*E_{(\lambda(p)|0)}, \pi_-^*\wedge^{q-1}\wt\Uc^\vee(-2))\arw 
        \Hc om(\pi_-^*\wedge^{p-1}\wt\Uc^\vee, \pi_-^*\wedge^{q-1}\wt\Uc^\vee)
    \end{equation}

    Let us give a very explicit description of $\Hc om( i_-^*E_{(\lambda(p)|0)},\pi_-^*\wedge^{q-1}\wt\Uc^\vee(-2))$:
    
    \begin{equation}
        \begin{split}
            \Hc om( i_-^*E_{(\lambda(p)|0)},\pi_-^*\wedge^{q-1}\wt\Uc^\vee(-2)) = \wt W\times\Hom(\wedge^p V_{n-1}^\vee, \wedge^{q-1}V_{n-2}\otimes(\wedge^{n-2}V_{n-2}^\vee)^{\otimes 2})/\operatorname{Stab}
        \end{split}
    \end{equation}

    Therefore, a section of such bundle will be defined by a $\operatorname{Stab}$-equivariant function
    
    \begin{equation}
        F : \wt W\arw \Hom(\wedge^p V_{n-1}^\vee, \wedge^{q-1}V_{n-2}\otimes(\wedge^{n-2}V_{n-2}^\vee)^{\otimes 2})
    \end{equation}
    
    where the equivariancy condition on $F$ is:
    
    \begin{equation}\label{eq:explicit_hom_section}
        F(v\det h^2+C\kappa, Ch^{-1}) = \wedge^p g F(v, C) \wedge^{q-1} h^{-1}.
    \end{equation}
    
    By Equation \ref{eq:stab} we can deduce the following decomposition:

    \begin{equation}
        \label{eq:wedgestab}
        \wedge^p g^{-1} = \left(\begin{matrix}
        \det h^2\wedge^{p-1}h^{-1} & 0 \\
        P(\kappa, h^{-1}) & \wedge^p h^{-1}
        \end{matrix}\right)\in GL(\wedge^p V_{n-1})
    \end{equation}
    
    where $P(\kappa, h^{-1})$ is the matrix consisting of those minors of order $p$ $(\kappa|h^{-1})$ which depend on the entries of $\kappa$.
    
    Note that $\wedge^p g$ decomposes as $\wedge^p g^{-1}$ does in Equation \ref{eq:wedgestab}. This allows to decompose $F$ of Equation \ref{eq:explicit_hom_section} in two functions $F_1$ and $F_2$, and the $\operatorname{Stab}$-action reads:
    
    \begin{equation}
        \begin{split}
            \left(
            \begin{array}{c}
                 F_1(v\det h^2+C\kappa, Ch^{-1})  \\
                 F_2(v\det h^2+C\kappa, Ch^{-1}) 
            \end{array}
            \right) = \hspace{250pt}\\
            = \left(
            \begin{array}{c}
                 \wedge^{p-1} h F_1(v, C)\wedge^{q-1}h^{-1}\\
                 P(\kappa, h^{-1})F_1(v, C)\wedge^{q-1}h^{-1}\det h^2 + \wedge^p h F_2(v,C) \wedge^{q-1}h\det h^2
            \end{array}
            \right).
        \end{split}
    \end{equation}

     From the expression above, we see $\beta_0$ as a map which sends a function $F$ as above (and hence the associated section) to its first component $F_1$. In the following, we will use the superscript $(0)$ to denote a function which is independent on $v$ (and therefore constant on fibers of $\pi_-$). Let us now consider a section $F$ such that $F_1$ is independent on $v$, i.e. $F_1 = F_1^{(0)}$. Then we have:
    
    \begin{equation}
        F =
        \left(
        \begin{array}{c}
            F_1^{(0)} \\
            F_2 
        \end{array}
        \right) =
        \left(
        \begin{array}{c}
            F_1^{(0)} \\
            F_2^{(0)} 
        \end{array}
        \right) +
        \left(
        \begin{array}{c}
            0 \\
            F_2^{(>0)} 
        \end{array}
        \right).
    \end{equation}
    
    Note that this decomposition is preserved by the action of $\operatorname{Stab}$, and this implies that both the summand of the right hand side are sections of $\Hc om( i_-^*E_{(\lambda(p)|0)},\pi_-^*\wedge^{q-1}\wt\Uc^\vee(-2))$. Since the second summand is clearly not in $\beta_0^{-1}(F_1)$, the proof reduces to show that $\Hc om( i_-^*E_{(\lambda(p)|0)},\pi_-^*\wedge^{q-1}\wt\Uc^\vee(-2))$ cannot have sections which are independent on $v$ and such that $F_1\neq 0$, or in other words:
    
    \begin{equation}
         \{F\in H^0(X_-, \Hc om(\pi_-^*\wedge^{p-1}\wt\Uc^\vee, \pi_-^*\wedge^{q-1}\wt\Uc^\vee) \hspace{2pt}:\hspace{2pt} F\text{\emph{ does not depend on }} v\}\nsubseteq\operatorname{im}\beta_0.
    \end{equation}
    
    In fact, from this and the long exact sequence of cohomology associated to the sequence \ref{eq:sequencewithH1} we would get that the left hand side of the equation above is isomorphic to $\operatorname{coker}\beta_0$, hence proving our claim.\\ 
    \\
    To this purpose, let us act on such $F$ with $g_\lambda: = \operatorname{diag}(\lambda^{-n+2}, \lambda, \dots, \lambda)$: the action on $F2 = F_2^{(0)}$ is
    
    \begin{equation}
       F_2^{(0)}(C\lambda^{-1}) = \lambda^{2n-3+p-q}F_2^{(0)}(C)
    \end{equation}
    
    Since $2n-3+p-q>0$ for $0\leq p\leq q\leq n-1$, $F_2^{(0)}$ is a matrix of polynomial, hence it must be the zero matrix. Let us now act on $F = (F_1^{(0)}, 0)$ by the full $\operatorname{Stab}$:
    
    \begin{equation}
        \left(
        \begin{array}{c}
             F_1^{(0)}(Ch^{-1})  \\
             0 
        \end{array}
        \right)
        = \left(
        \begin{array}{c}
             \wedge^{p-1} h F_1^{(0)}(C)\wedge^{q-1}h^{-1}\\
             P(\kappa, h^{-1})F_1^{(0)}(C)\wedge^{q-1}h^{-1}\det h^2
        \end{array}
        \right).
    \end{equation}
    
    Since $\wedge^{q-1}h^{-1}$ is invertible, the condition $P(\kappa, h^{-1})F_1^{(0)}(C)\wedge^{q-1}h^{-1}\det h^2 = 0$ is fulfilled if and only if $P(\kappa, h^{-1})F_1^{(0)}(C) = 0$, i.e. the columns of $F_1^{(0)}(C)$ must lie in the kernel of $P(\kappa, h^{-1})$ for every $\kappa\in V_{n-2}$. But varying $\kappa$, the kernels of the associated $P(\kappa, h^{-1})$ span the whole domain and therefore $F_1^{(0)} = 0$, and this concludes the proof.
\end{proof}

\begin{lemma}\label{lem:difficultvanishing1}
    The following vanishing holds for $0\leq p, q\leq n-3$:
    
    $$
    \Ext^{>0}( i_-^*E_{(\lambda(p)|0)},  i_-^*E_{(\lambda(q)|0)})) = 0.
    $$
    
\end{lemma}

\begin{proof}
    One has the following diagram, which is a consequence of the sequence of Equation \ref{eq:vgitsequence} and its dual:
    
    \begin{equation}
    \small
        \begin{tikzcd}[row sep = huge, column sep = small]
            \pi_-^*(\wedge^{n-p-2}\wt\Uc^\vee\otimes\wedge^{q-1}\wt\Uc^\vee(-3))\ar[hook]{d} & & \pi_-^*(\wedge^{n-p-2}\wt\Uc^\vee\otimes\pi_-^*\wedge^q\wt\Uc^\vee(-1))\ar[hook]{d}\\
            ( i_-^*E_{(\lambda(p)|0)})^\vee\otimes\pi_-^*\wedge^{q-1}\wt\Uc^\vee(-2)\ar[two heads]{d}\ar[hook]{r} & ( i_-^*E_{(\lambda(p)|0)})^\vee\otimes i_-^*E_{(\lambda(q)|0)}) \ar[two heads]{r} & ( i_-^*E_{(\lambda(p)|0)})^\vee\otimes\pi_-^*\wedge^q\wt\Uc^\vee)\ar[two heads]{d}\\
            \pi_-^*(\wedge^{n-p-1}\wt\Uc^\vee\otimes\wedge^{q-1}\wt\Uc^\vee(-1)) & & \pi_-^*(\wedge^{n-p-1}\wt\Uc^\vee\otimes\pi_-^*\wedge^q\wt\Uc^\vee(1))
        \end{tikzcd}
    \end{equation}

    Since $\pi_-^*(\wedge^{n-p-1}\wt\Uc^\vee\otimes\wedge^q\wt\Uc^\vee(1))$ is a tensor product of globally generated vector bundles it decomposes in a direct sum of globally generated, irreducible homogeneous vector bundles, and therefore it has no higher cohomology.\\
    \\
    Both $\wedge^{n-p-1}\wt\Uc^\vee\otimes\wedge^{q-1}\wt\Uc^\vee(-1)$ and $\wedge^{n-p-2}\wt\Uc^\vee\otimes\wedge^q\wt\Uc^\vee(-1)$ decompose as direct sums of bundles associated to weights of the form $(\lambda_1, \dots, \lambda_{n-2}|1,1)$ with $\lambda_i\leq 2$ for every $i$. The product of such a summand with $\Sym^m\Qc^\vee(2m)$ can be described by the following diagrams:
    
    
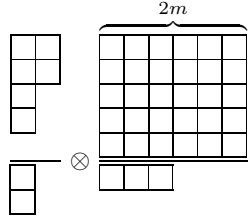
\begin{equation}
            \begin{split}
                \tiny
                \frac{\ydiagram{2,2,1,1,0}}{\hspace{-10pt}\ydiagram{1,1}}\otimes
                \frac{\overbrace{\ydiagram{6,6,6,6,6}}^{2m}}{\hspace{-28pt}\ydiagram{3}}
            \end{split}
    \end{equation}
    
    Clearly, if $\lambda_{n-2} = 0$ there is no cohomology for $m=0$ and only sections for $m>0$, while for $\lambda_{n-2}>0$ there are again only sections. For further relevance, we note that the only contribution with $m=0$ to the sections of $\wedge^{n-p-1}\wt\Uc^\vee\otimes\wedge^{q-1}\wt\Uc^\vee(-1)$ is a vector space isomorphic to $\wedge^{q-p}V_n$: in other words, there is such a vector space of global sections of $\wedge^{n-p-1}\wt\Uc^\vee\otimes\wedge^{q-1}\wt\Uc^\vee(-1)$ which are constant on the fibers.\\
    \\
    The bundle $\pi_-^*(\wedge^{n-p-2}\wt\Uc^\vee\otimes\wedge^{q-1}\wt\Uc^\vee(-3))$ requires a more careful analysis. Let us first consider the contribution for $m = 0$: $\wedge^{n-p-2}\wt\Uc^\vee\otimes\wedge^{q-1}\wt\Uc^\vee(-3)$ decomposes in direct summands which can have the following shapes:
    
    
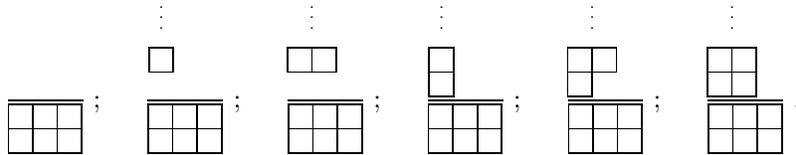
\begin{equation}
            \begin{split}
                \tiny
                \frac{\begin{array}{c}
                    \hspace{-18pt}\vspace{5pt}\\
                    \ydiagram{0,0}
                \end{array}}{\ydiagram{3,3}}\hspace{3pt};\hspace{15pt}
                \frac{\hspace{-18pt}\begin{array}{c}
                    \vdots\vspace{5pt}\\
                    \ydiagram{1,0}
                \end{array}}{\ydiagram{3,3}}\hspace{3pt};\hspace{15pt}
                \frac{\hspace{-10pt}\begin{array}{c}
                    \vdots\vspace{5pt}\\
                    \ydiagram{2,0}
                \end{array}}{\ydiagram{3,3}}\hspace{3pt};\hspace{15pt}\frac{\hspace{-18pt}\begin{array}{c}
                    \vdots\vspace{5pt}\\
                    \ydiagram{1,1}
                \end{array}}{\ydiagram{3,3}}\hspace{3pt};\hspace{15pt}
                \frac{\hspace{-10pt}\begin{array}{c}
                    \vdots\vspace{5pt}\\
                    \ydiagram{2,1}
                \end{array}}{\ydiagram{3,3}}\hspace{3pt};\hspace{15pt}
                \frac{\hspace{-10pt}\begin{array}{c}
                    \vdots\vspace{5pt}\\
                    \ydiagram{2,2}
                \end{array}}{\ydiagram{3,3}}\hspace{3pt}.
            \end{split}
    \end{equation}
    
    All of them give repetitions, except for the third one. However, that configuration can happen only if $n-p-2 = q-1 = n-3$, which implies $q = n-2$, and this is against our assumptions.
    
    Consider now the summands with $m=1$. Here the only possibilities are
    
    \begin{equation}
            \begin{split}
                \tiny
                \frac{\hspace{-18pt}\begin{array}{c}
                    \vspace{5pt}\\
                    \ydiagram{2,2}
                \end{array}}{\ydiagram{4,3}}\hspace{3pt};\hspace{15pt}
                \frac{\hspace{-10pt}\begin{array}{c}
                    \vdots\vspace{5pt}\\
                    \ydiagram{3,2}
                \end{array}}{\ydiagram{4,3}}\hspace{3pt};\hspace{15pt}
                \frac{\hspace{-0pt}\begin{array}{c}
                    \vdots\vspace{5pt}\\
                    \ydiagram{4,2}
                \end{array}}{\ydiagram{4,3}}\hspace{3pt};\hspace{15pt}\frac{\hspace{-10pt}\begin{array}{c}
                    \vdots\vspace{5pt}\\
                    \ydiagram{3,3}
                \end{array}}{\ydiagram{4,3}}\hspace{3pt};\hspace{15pt}
                \frac{\hspace{-0pt}\begin{array}{c}
                    \vdots\vspace{5pt}\\
                    \ydiagram{4,3}
                \end{array}}{\ydiagram{4,3}}\hspace{3pt};\hspace{15pt}
                \frac{\hspace{-0pt}\begin{array}{c}
                    \vdots\vspace{5pt}\\
                    \ydiagram{4,4}
                \end{array}}{\ydiagram{4,3}}\hspace{3pt}.
            \end{split}
    \end{equation}
    
    Clearly, higher cohomology can occur only in the second and third configurations, but the fourth one never happens for $0\leq p,q\leq n-3$ as we argued above. However, the second one does appear whenever $p\leq q$, yielding a contribution of $\wedge^{q-p}V_n$ to $H^1(X_-, \pi_-^*(\wedge^{n-p-2}\wt\Uc^\vee\otimes\wedge^{q-1}\wt\Uc^\vee(-3)))$. A quick check for $m>1$ ensures that there are no additional contributions to higher cohomology for this bundle. This, together with the fact that the only $m=0$ contribution to $H^0(X_-,\wedge^{n-p-1}\wt\Uc^\vee\otimes\wedge^{q-1}\wt\Uc^\vee(-1))$ is isomorphic to $\wedge^{q-p}V_n$, allows to conclude by Lemma \ref{lem:lifesavinglemma}.
\end{proof}

\begin{lemma}\label{lem:difficultvanishing2}
    One has $\Ext^{>0}(  i_-^*E_{(k, 0,\dots,0|0)},  i_-^*E_{(\lambda(p)|0)}) = \Ext^{>0}( i_-^*E_{(\lambda(p)|0)},  i_-^*E_{(k, 0,\dots,0|0)}) = 0$ for $0\leq p\leq n-3$ and $0\leq m\leq 2$.
\end{lemma}

\begin{proof}
    Recall that $ i_-^*E_{(k, 0,\dots,0|0)} = \Oc(-k)$: proving our claim amounts to show that $ i_-^*E_{(\lambda(p)|0)}\otimes\Oc(k)$ and $( i_-^*E_{(\lambda(p)|0)})^\vee\otimes\Oc(-k)$ have no higher cohomology. The first bundle is realized as an extension in:
    
    \begin{equation}
        0\arw\wedge^{p-1}\wt\Uc^\vee(k-2)\arw i_-^*E_{(\lambda(p)|0)}\otimes\Oc(k)\arw\wedge^p\wt\Uc^\vee(k)\arw 0.
    \end{equation}
    
    Here the third bundle is the pullback of a homogeneous, irreducible and globally generated bundle, and hence it has no cohomology in higher degree. Consider now $\wedge^{p-1}\wt\Uc^\vee(k-2)\otimes\Sym^m\wt\Qc^\vee(2m)$ for $m\geq 0$. Since $p\leq n-3$, the last entries of the associated weight will be
    
    $$(\dots, 2m, 2m | m-k+2, -k+2)$$
    
    which always yield a repetition or sections.\\
    \\
    On the other hand, the cohomology of $( i_-^*E_{(\lambda(p)|0)})^\vee\otimes\Oc(-k)$ is computed by
    
    \begin{equation}
        0\arw \wedge^{n-p-2}\wt\Uc^\vee(-1-k)\arw( i_-^*E_{(\lambda(p)|0)})^\vee\arw \wedge^{n-p-1}\Uc^\vee(1-k)\arw 0.
    \end{equation}
    
    The third is the pullback of a globally generated bundle for $k\leq 1$, while for $k = 2$ the weight of any direct summand of its pushforward is $(\dots\lambda_{n-4}+2m, \lambda_{n-3}+2m | m+1, 1)$, which again cannot have higher cohomology. Let us now turn our attention to $\wedge^{n-p-2}\wt\Uc^\vee(-1-k)$. Here all direct summands of the pushforward will have weight $(\dots\lambda_{n-4}+2m, \lambda_{n-3}+2m | m+1+k, 1+k)$. A careful analysis permits to conclude that the only higher cohomology occurring is $H^1\simeq \CC$ for $k = 2, p = 1, m=1$. We can treat this situation exactly as we did in Lemma 
    \ref{lem:difficultvanishing1}, once we observe that the contribution for $m=0$ to $H^0(X_-, \wedge^{n-2}\Uc^\vee(-1))$ is one dimensional. The rest of the proof follows exactly as in the previous lemma, i.e. by applying Lemma \ref{lem:lifesavinglemma} for $q=0, p=1$.
\end{proof}

Summing all up, we find:

\begin{proposition}\label{prop:functors_no_potential}
    There is a fully faithful functor $\dbcoh(X_{ {+}})\xrightarrow{\,\,\,\,\,\,}\dbcoh(X_{ {-}})$.
\end{proposition}

\begin{proof}

By Proposition \ref{prop:equivalence+_no_potential} we already know that $i_+|_{\Wc_0}^*$ is an equivalence, while by Proposition \ref{prop:partiallytilting-} $i_-|_{\Wc_0}^*$ is fully faithful
hence we get the desired result in the form of a functor 

\begin{equation}
    i_-|_{\Wc_0}^* \circ (i_+|_{\Wc_0}^*)^{-1}:\dbcoh(X_+)\arw\dbcoh(X_-).
\end{equation}

\end{proof}

\subsection{Turn on the superpotential: B-brane categories}\label{subsection:branes}
The goal of this section is to lift the derived embedding  {$\dbcoh X_+\subset\dbcoh X_-$} to a fully faithful functor of B-brane categories. We will slightly specialize the definitions of \cite[Section 2]{segal} to the present setting and terminology. In the following, let $X$ be a smooth scheme (or stack), together with a function (superpotential) $f:X\arw\CC$. For the moment we do not need to have a $G$-action on $X$ which leaves $f$ invariant, but just a $\CC^*$-action on $X$ with weight two on $f$ and such that $-1\in\CC^*$ acts trivially on $X$. As above, we denote this action ``R-symmetry'' and we use the traditional notation $\CC^*_R$. We will denote this data by the expression $(X, f)$, without explicitly referring to the R-symmetry.

\begin{definition}\label{def:brane}
     We call \emph{B-brane} on $(X, f)$ a pair $(E, d_E)$ where $E$ is a vector bundle on $X$ and $d_E$ is an endomorphism of $E$ with $\CC^*_R$-weight one, such that $d^2 = \II_E f$, where $\II_E\in\Hom_X(E, E)$ is the identity morphism.
\end{definition}

\begin{definition}\label{def:branedgcat}
    We call $\Bc r(X, f)$ the category whose objects are B-branes on $(X, f)$ and Hom-sets are defined as
    \begin{equation}
        \Hom_{\Bc r(X, f)}((E,d_E), (F, d_F)) := (\Hc om_X(E, F), d_{E,F})
    \end{equation}
    where $d_{E,F}:=\II_E^\vee\otimes d_F - d_E^\vee\otimes\II_F$.
\end{definition}

Note that $\Bc r(X, f)$ is not a dg-category, but rather it is enriched over the category of B-branes on the model $(X, G, \CC^*_R, G, 0)$, i.e. with a trivial superpotential (this category has been defined in \cite[Definition 2.7]{segal}). Now we are ready to recall the definition of B-brane category:

\begin{definition}\label{def:branecat}
    We call $\operatorname{Br}(X, f)$ the category whose objects are B-branes on $(X, f)$ and Hom-sets are defined as
    \begin{equation}
        \Hom_{\operatorname{Br}(X, f)}((E,d_E), (F, d_f)) = R\Gamma((\Hc om_X(E, F), d_{E,F})),
    \end{equation}
    where the functor $R\Gamma$ on a B-brane is defined in \cite[Definition 2.7]{segal}.
\end{definition}

Computing morphisms in $\operatorname{Br}(X, f)$ is simplified by the existence of a spectral sequence which degenerates at its first page:

\begin{equation}\label{eq:spectralsequence}
    \left( H^p(X, \Hc om_X(E, F)), d_{E,F}\right) \xRightarrow{} R\Gamma((\Hc om_X(E, F), d_{E,F})).
\end{equation}

Here the vertical grading on the right hand side is given by the $R$-symmetry.

Let us denote by $\Wc$ the subcategory of $\operatorname{Br}(\Xc, f)$ generated as direct sums by the objects

 {
$$\left\{E_{(\lambda|0)} : \lambda\in I_W\right\}$$
}

together with appropriate endomorphisms $d_{E_\lambda}$.  {Recall that the same vector bundles generate $\Wc_0$.}

\begin{proposition}\label{prop:ff_potential}
    The functors $\iota_\pm^*|_\Wc: \Wc\arw \operatorname{Br}(X_\pm, f)$ are fully faithful.
\end{proposition}

\begin{proof}
    In light of the spectral sequence \ref{eq:spectralsequence}, we just need to compare the following bicomplexes:
    
    \begin{equation}
        \left( H^p(X_\pm, \Hc om_{X_\pm}(\iota_\pm^*E, \iota_\pm^*F)), d_{\iota_\pm^*E,\iota_\pm^*F}\right),
    \end{equation}
    
    \begin{equation}
        \left( H^p(\Xc, \Hc om_\Xc(E, F)), d_{E,F}\right).
    \end{equation}
    
    Both of them are nonzero only on the row $p=0$. The claim for the first one simply follows from the fact that there are no higher Exts between vector bundles on $\Xc$, while for the first one the statement is just a consequence of Proposition \ref{prop:functors_no_potential}. By Hartogs's lemma we conclude that:
    
    \begin{equation}
        H^p(X_\pm, \Hc om_{X_\pm}(\iota_\pm^*E, \iota_\pm^*F)) \simeq H^p(\Xc, \Hc om_\Xc(E, F)).
    \end{equation}
    
    therefore, we get an isomorphism between the first pages of two spectral sequences, which in turn gives an isomorphism $R\Gamma((\Hc om_{X_\pm}(\iota_\pm^*E, \iota_\pm^*F), d_{\iota_\pm^*E,\iota_\pm^*F})) \simeq R\Gamma((\Hc om_\Xc(E, F), d_{E,F})).$
\end{proof}

\begin{proposition}\label{prop:ess_surj_potential}
     {$\iota_+^*|_\Wc$} is essentially surjective.
\end{proposition}

\begin{proof}
    The proof is an adaptation of the one of \cite[Lemma 3.6]{segal} to the present setting, hence we will keep it concise. Let us consider a brane  {$(E, d_E)\in\operatorname{Br}(X_+, f)$}. Then, since $E$ is a vector bundle on  {$X_+$}, it admits a resolution of the form
    
    \begin{equation}\label{eq:resolution}
        \begin{tikzcd}
            0\ar{r} & E_{-s}\ar{r}{\partial_\Ec} & E_{-s+1}\ar{r}{\partial_\Ec} & \cdots\ar{r}{\partial_\Ec} & E_{-1}\ar{r}{\partial_\Ec} & E_0\ar{r}{q} & E\ar{r} & 0
        \end{tikzcd}
    \end{equation}

    where every $E_i$ is a direct sum of bundles of the form  {$\Sc_\lambda \Uc^\vee$} with  {$\lambda\in I_W$}. Out of this resolution, we can construct a complex
    
    \begin{equation}
        \Ec:=\bigoplus_{i}E_{-i}[i].
    \end{equation}

    Observe that $\partial_\Ec$, as an endomorphism of $\Ec$, has $\CC^*_R$-weight one.
    
    \textbf{Claim.} \emph{There exists a map $d_\Ec$ with the property $d_\Ec^2 = If$ and such that the corresponding brane $(\Ec, d_\Ec)$ is homotopy-equivalent to $(E, d_E)$.}
    
    Suppose this claim is true. Call $\wh\Ec$ the complex of vector bundles on $\Xc$ such that $\Ec$ is its restriction to $X_-$. 
    As noted in the proof of Proposition \ref{prop:functors_no_potential}, it follows by Hartogs's lemma that  {$H^{0}(X_+, \Hc om_{X_+}(\Ec, \Ec)) = H^{0}(\Xc, \Hc om_\Xc(\wh\Ec, \wh\Ec))$}. In particular, $d_\Ec$ is the restriction of an endomorphism $d_{\wh\Ec}$ of $\wh\Ec$: this proves that every brane $(E, d_E)$ in  {$\operatorname{Br}(X_+, f)$} is homotopically equivalent (and hence identified in  {$\operatorname{Br}(X_+, f)$}) to the restriction to  {$X_+$} of a brane $(\wh\Ec, d_{\wh\Ec})$ in $\operatorname{Br}(\Xc, f)$, and this concludes the proof.
    
    The proof of the Claim is rather technical, but it is identical to the second part of \cite[proof of Lemma 3.6]{segal}, hence it will be omitted.
\end{proof}

\begin{theorem}\label{thm:derived_embedding_main}
    Let $S\in H^0(F(n-2, n-1, n), \Oc(1,1))$ be a general section, fix $M:= Z(\sigma)$ and $Y_i:= Z(h_{i*}(S)$ for $i\in\{1;2\}$. Then there is a fully faithful functor $\dbcoh(Y_2)\xhookrightarrow{\,\,\,\,\,\,} \dbcoh(Y_1)$.
\end{theorem}

\begin{proof}
    By Propositions \ref{prop:ff_potential}, \ref{prop:ess_surj_potential} there is a fully faithful functor  {$\operatorname{Br}(X_+, f)\xhookrightarrow{\,\,\,\,\,\,}\operatorname{Br}(X_-, f)$}. By Kn\"orrer periodicity \cite[Theorem 3.4]{shipman}, there are equivalences $
    \dbcoh(\operatorname{Crit}(f)\git_{\pm}G)\simeq \operatorname{Br}(X_\pm, f)$. The proof follows by the isomorphisms  {$\operatorname{Crit}(f)\git_+G\simeq Y_1$} and  {$\operatorname{Crit}(f)\git_-G\simeq Y_2$} developed in Lemma \ref{lem:glsm_crit_loci}.
\end{proof}

\begin{remark}\label{rem:chessgame}
    Note that in \cite{leungxie} an embedding  {$\dbcoh(X_+)\subset\dbcoh(X_-)$} is constructed with a completely different method: in fact Leung and Xie provide semiorthogonal decompositions of the total space of $\Oc(-1,-1)$ on $F(1,2,V)$ and by an inductive method based on mutations of exceptional object they prove that  {$\dbcoh(X_+)$} is equivalent to a subcategory of  {$\dbcoh(X_-)$}, hence proving the DK conjecture for this example. However, their result cannot be directly applied to our method since it is not compatible with the construction of the window category. However, one could proceed in the spirit of \cite{leungxie} by constructing semiorthogonal decompositions of the zero locus $M$ of a general section of $\Oc(1,1)$ by means of \cite{cayleytrick}. These decompositions are formally identical to the ones given by \cite{leungxie}, and a similar approach based on mutations is a possible way to recover the embedding $\dbcoh(Y_1)\subset\dbcoh(Y_2)$. However, the current proof, despite providing the embedding in a very abstract fashion, is considerably shorter and more suitable to generalizations.
\end{remark}

The authors have no competing interests to declare that are relevant to the content of this article. Data sharing not applicable to this article as no datasets were generated or analysed during the current study. Corresponding author: Marco Rampazzo.

\end{document}